%% file: extremal_regular.tex
\documentclass[11pt, reqno]{amsart}
\usepackage[T1]{fontenc}
\usepackage[margin=1in]{geometry}
\usepackage{microtype}

\usepackage{amsthm, amsmath,amssymb,tikz}
\usetikzlibrary{calc}

\usepackage[hidelinks]{hyperref}

\newtheorem{theorem}{Theorem}[section]

\newtheorem{lemma}[theorem]{Lemma}

\newtheorem{corollary}[theorem]{Corollary}
\newtheorem{conjecture}[theorem]{Conjecture}
\newtheorem*{question*}{Question}

\theoremstyle{definition}
\newtheorem{definition}[theorem]{Definition}
\newtheorem{problem}[theorem]{Problem}
\newtheorem{question}[theorem]{Question}
\newtheorem*{definition*}{Definition}
\newtheorem{example}[theorem]{Example}

\theoremstyle{remark}

\newcommand{\abs}[1]{\left\lvert#1\right\rvert}

\newcommand{\wt}{\widetilde}

\newcommand{\ol}{\overline}

\DeclareMathOperator{\Hom}{Hom}
\DeclareMathOperator{\area}{area}
\DeclareMathOperator{\vol}{vol}

\newcommand{\ind}{\mathrm{ind}}

\newcommand{\bad}{\mathrm{bad}}
\newcommand{\bst}{\mathrm{bst}}

\newcommand{\x}{\times}

\newcommand{\EE}{\mathbb{E}}

\newcommand{\RR}{\mathbb{R}}
\newcommand{\PP}{\mathbb{P}}

\newcommand{\cF}{\mathcal{F}}

\newcommand{\cG}{\mathcal{G}}

\newcommand{\bx}{\mathbf{x}}

\tikzstyle{B}=[draw,circle, fill=black, minimum size=6pt,inner sep=0pt]
\tikzstyle{W}=[draw,circle, fill=white, minimum size=6pt,inner sep=0pt]
\tikzstyle{P} = [draw, circle, black, fill, inner sep = 0pt, minimum width = 3pt]
\tikzstyle{every loop} = []

\usetikzlibrary{shapes}
\tikzstyle{c1}=[shape=rectangle, minimum size=10pt, fill=red]
\tikzstyle{c2}=[shape=diamond, minimum size=8pt, fill=blue]
\tikzstyle{c3}=[shape=circle, minimum size=8pt, fill=green]

\newcommand{\tikzHind}{
  \begin{tikzpicture}[baseline, yshift=1pt]
    \path[use as bounding box] (-.15,-.1) rectangle (.6,.35);
    \draw (0.5,0) node[P] {} -- (0,0) node[P] {} edge[-,in = 45, out = 135, loop] ();
  \end{tikzpicture}
}

\newcommand{\tikzHwr}{
  \begin{tikzpicture}[baseline,yshift=1pt]
    \path[use as bounding box] (-.15,-0.1) rectangle (0.95,.35);
    \draw (0,0) node[P] {} edge[-,in = 45, out = 135, loop] () 
      -- (.4,0) node[P] {} edge[-,in = 45, out = 135, loop] ()
      -- (.8,0) node[P] {} edge[-,in = 45, out = 135, loop] ();
  \end{tikzpicture}
}

\newcommand{\tikzHwrdown}{
  \begin{tikzpicture}[baseline,yshift=5pt]
    \path[use as bounding box] (-.15,-0.1) rectangle (0.95,.35);
    \node[P] (a) at (0,0) {};
    \node[P] (b) at (.4,0) {};
    \node[P] (c) at (.8,0) {};
    \node[P] (d) at (.4,-.3) {};
    \draw (a) edge[-,in = 45, out = 135, loop] () 
       -- (b) edge[-,in = 45, out = 135, loop] ()
       -- (c) edge[-,in = 45, out = 135, loop] ()
          (b) -- (d);
  \end{tikzpicture}
}

\newcommand{\tikzHtwoloops}{
  \begin{tikzpicture}[baseline,yshift=1pt]
    \path[use as bounding box] (-.15,-0.1) rectangle (0.55,.35);
    \draw (0,0) node[P] {} edge[-,in = 45, out = 135, loop] () 
         (.4,0) node[P] {} edge[-,in = 45, out = 135, loop] ();
  \end{tikzpicture}
}

\newcommand{\tikzKtwo}{
  \begin{tikzpicture}[baseline,yshift=0.5ex]
    \path[use as bounding box] (-.15,-.1) rectangle (.6,.1);
    \draw (0.5,0) node[P] {} -- (0,0) node[P] {};
  \end{tikzpicture}
}

\newcommand{\tikzHindsquared}{
  \begin{tikzpicture}[baseline,yshift=-.5ex]
    \node[P] (a) at (0,0) {};
    \node[P] (b) at (.4,0) {};
    \node[P] (c) at (0,.4) {};
    \node[P] (d) at (.4,.4) {};
    \draw (a) edge[-,in = 135, out = 225, loop] () 
       -- (b)--(c)--(a)--(d);
  \end{tikzpicture}
}

\title[Extremal regular graphs]{Extremal regular graphs: \\ independent sets and graph homomorphisms}
\author{Yufei Zhao}
\address{Mathematical Institute, Oxford OX2 6GG, United Kingdom}
\email{yufei.zhao@maths.ox.ac.uk}
\thanks{The author is supported by an Esm\'ee Fairbairn 
Junior Research Fellowship at New College, Oxford.}

\begin{document}

\begin{abstract}
This survey concerns regular graphs that are extremal with respect to the number of independent sets, and more generally, graph homomorphisms. More precisely, in the family of of $d$-regular graphs, which graph $G$ maximizes/minimizes the quantity $i(G)^{1/v(G)}$, the number of independent sets in $G$ normalized exponentially by the size of $G$? What if $i(G)$ is replaced by some other graph parameter?
We review existing techniques, highlight some exciting recent developments, and discuss open problems and conjectures for future research.
\end{abstract}

\maketitle

\section{Independent sets} \label{sec:ind}

An \emph{independent set} in a graph is a subset of vertices with no two adjacent. Many combinatorial problems can be reformulated in terms of independent sets by setting up a graph where edges represent forbidden relations. 

A graph is \emph{$d$-regular} if all vertices have degree $d$. In the family of $d$-regular graphs of the same size, which graph has the most number of independent sets? This question was initially raised by Andrew Granville in connection to combinatorial number theory, and appeared first in print in a paper by Alon~\cite{Alon91}, who speculated that, at least when $n$ is divisible by $2d$, the maximum is attained by a disjoint union of complete bipartite graphs $K_{d,d}$. Some ten years later, Kahn~\cite{Kahn01} arrived at the same conjecture while studying a problem arising from statistical physics. Using a beautiful entropy argument, Kahn proved the conjecture under the additional assumption that the graph is already bipartite.

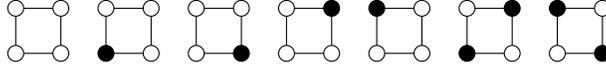
\begin{figure}[t]
\centering
\begin{tikzpicture}[scale=.6]
\begin{scope}
	\draw  (0,0) node[W] {} -- 
	       (1,0) node[W] {} -- 
	       (1,1) node[W]{} --
	       (0,1) node[W]{} -- cycle;
\end{scope}

\begin{scope}[shift={(2,0)}]
	\draw  (0,0) node[B] {} -- 
	       (1,0) node[W] {} -- 
	       (1,1) node[W]{} --
	       (0,1) node[W]{} -- cycle;
\end{scope}

\begin{scope}[shift={(4,0)}]
	\draw  (0,0) node[W] {} -- 
	       (1,0) node[B] {} -- 
	       (1,1) node[W]{} --
	       (0,1) node[W]{} -- cycle;
\end{scope}

\begin{scope}[shift={(6,0)}]
	\draw  (0,0) node[W] {} -- 
	       (1,0) node[W] {} -- 
	       (1,1) node[B]{} --
	       (0,1) node[W]{} -- cycle;
\end{scope}

\begin{scope}[shift={(8,0)}]
	\draw  (0,0) node[W] {} -- 
	       (1,0) node[W] {} -- 
	       (1,1) node[W]{} --
	       (0,1) node[B]{} -- cycle;
\end{scope}

\begin{scope}[shift={(10,0)}]
	\draw  (0,0) node[B] {} -- 
	       (1,0) node[W] {} -- 
	       (1,1) node[B]{} --
	       (0,1) node[W]{} -- cycle;
\end{scope}

\begin{scope}[shift={(12,0)}]
	\draw  (0,0) node[W] {} -- 
	       (1,0) node[B] {} -- 
	       (1,1) node[W] {} --
	       (0,1) node[B] {} -- cycle;
\end{scope}
\end{tikzpicture}
\caption{The independent sets of a 4-cycle: $i(C_4) =7$.} \label{fig:indcount}
\end{figure}

We write $I(G)$ to denote the set of independent sets in $G$, and $i(G) := |I(G)|$ the number of independent sets in $G$. See Figure~\ref{fig:indcount}.

\begin{theorem}[Kahn~\cite{Kahn01}] \label{thm:kahn}
  If $G$ is a bipartite $n$-vertex $d$-regular graph, then
  \[
    i(G) \le i(K_{d,d})^{n/(2d)} = (2^{d+1} - 1)^{n/(2d)}.
  \]
\end{theorem}

I showed that the bipartite requirement in Theorem~\ref{thm:kahn} can be dropped (as conjectured by Kahn).

\begin{theorem} [\cite{Zhao10}] \label{thm:zhao}
  If $G$ is an $n$-vertex $d$-regular graph, then
  \[
    i(G) \le  i(K_{d,d})^{n/(2d)} = (2^{d+1} - 1)^{n/(2d)}.
  \]
\end{theorem}

Equality occurs when $n$ is divisible by $2d$ and $G$ is a disjoint union of $K_{d,d}$'s. We do not concern ourselves here with what happens when $n$ is not divisible by $2d$, as the extremal graphs are likely dependent on number theoretic conditions, and we do not know a clean set of examples. Alternatively, the problem can phrased as maximizing $i(G)^{1/v(G)}$ over the set of $d$-regular bipartite graphs $G$, where $v(G)$ denotes the number of vertices of $G$. The above theorem says that this maximum is attained at $G = K_{d,d}$. Note that $i(G)^{1/v(G)}$ remains unchanged if $G$ is replaced by a disjoint union of copies of $G$.

We provide an exposition of the proofs of these two theorems as well as a discussion of subsequent developments. Notably, Davies, Jenssen, Perkins, and Roberts~\cite{DJPR1} recently gave a new proof of the above theorems by introducing a powerful new technique, which has already had a number of surprising new consequences~\cite{DJPR2,DJPR3,PP}. The results have been partially extended to graph homomorphisms, though many intriguing open problems remain. We also discuss some recent work on the subject done by Luke Sernau~\cite{Ser} as an undergraduate student at Notre Dame.

\section{Graph homomorphisms} \label{sec:hom}

Given two graphs $G$ and $H$, a \emph{graph homomorphism} from $G$ to $H$ is a map of vertex sets $\phi \colon V(G) \to V(H)$ that sends every edge of $G$ to an edge of $H$, i.e., $\phi(u)\phi(v) \in E(H)$ whenever $uv \in E(G)$. Here $V(G)$ denotes the vertex set of $G$ and $E(G)$ the edge set. We use lower case letters for cardinalities: $v(G) := |V(G)|$ and $e(G) := |E(G)|$. Let
\[
\Hom(G,H) := \{ \phi \colon V(G) \to V(H) : \phi(u)\phi(v) \in E(H) \ \forall uv \in E(G)\}
\]
denote the set of graph homomorphisms from $G$ to $H$, and $\hom(G,H) := |\Hom(G,H)|$.

We usually use the letter $G$ for the source graph and $H$ for the target graph. It will be useful to allow the target graph $H$ to have loops (but not multiple edges), and we shall refer to such graphs as \emph{loop-graphs}. The source graph $G$ is usually simple (without loops). By \emph{graph} we usually mean a simple graph.

Graph homomorphisms generalize the notion of independent sets. They are equivalent to labeling the vertices of $G$ subject to certain constraints encoded by $H$.

\begin{figure}
\centering
\begin{tikzpicture}[scale=.6, >=latex]
\begin{scope}[shift={(-8,0)}]
\node[P] (1) at (90:1) {};
\node[P] (2) at (162:1) {};
\node[P] (3) at (234:1) {};
\node[P] (4) at (306:1) {};
\node[P] (5) at (378:1) {};
\draw (1)--(2)--(3)--(4)--(5)--(1);
\end{scope}

\begin{scope}[shift={(-2,0)}, font={\footnotesize}]
\node[B] (b) at (1,0) {};
\node[W] (w) at (0,0) {};
\draw (b)--(w) edge[-,in=45,out=135,loop] ();
\end{scope}

\begin{scope}[-latex, shorten <=3pt, shorten >=3pt]
\draw (1) to[bend left] (b);
\draw (2) to[bend right=15] (w);
\draw (3) to[bend right=20] (b);
\draw (4) to[bend right=15] (w);
\draw (5) to[bend left=5] (w);
\end{scope}

\begin{scope}[shift={(3,0)}, font=\footnotesize]
\node[B] (1) at (90:1) {};
\node[W] (2) at (162:1) {};
\node[B] (3) at (234:1) {};
\node[W] (4) at (306:1) {};
\node[W] (5) at (378:1) {};
\draw (1)--(2)--(3)--(4)--(5)--(1);
\end{scope}

\node[inner sep=1em] (label-left) at (-5, -2)  {graph homomorphism }; 
\node[inner sep=1em]  (label-right) at (3,-2) {independent set };
\draw[<->] (label-left) to (label-right);

\end{tikzpicture}
\caption[Graph homomorphisms and independent sets]{Homomorphisms from $G$ to $\tikzHind$ correspond to independent sets of $G$.}
\label{fig:ind}
\end{figure}

\begin{figure}
\centering
\begin{tikzpicture}[scale=.6, >=latex]
\begin{scope}[shift={(-8,0)}]
\node[P] (1) at (90:1) {};
\node[P] (2) at (162:1) {};
\node[P] (3) at (234:1) {};
\node[P] (4) at (306:1) {};
\node[P] (5) at (378:1) {};
\draw (1)--(2)--(3)--(4)--(5)--(1);
\end{scope}

\begin{scope}[shift={(-2,0)}]
\node[c1] (a) at (90:1) {};
\node[c2] (b) at (210:1) {};
\node[c3] (c) at (330:1) {};
\draw (a)--(b)--(c)--(a);
\end{scope}

\begin{scope}[-latex, shorten <=3pt, shorten >=3pt]
\draw (1) to[bend left=5] (b);
\draw (2) to[bend left=5] (a);
\draw (3) to[bend left=10] (b);
\draw (4) to[bend right=15] (c);
\draw (5) to[bend left=5] (a);
\end{scope}

\begin{scope}[shift={(3,0)}]
\node[c2] (1) at (90:1) {};
\node[c1] (2) at (162:1) {};
\node[c2] (3) at (234:1) {};
\node[c3] (4) at (306:1) {};
\node[c1] (5) at (378:1) {};
\draw (1)--(2)--(3)--(4)--(5)--(1);
\end{scope}

\node[inner sep=1em] (label-left) at (-5, -2)  {graph homomorphism}; 
\node[inner sep=1em]  (label-right) at (3,-2) {coloring};
\draw[<->] (label-left) to (label-right);

\end{tikzpicture}
\caption{Homomorphisms from $G$ to $K_q$ correspond to proper colorings of vertices of $G$ with $q$ colors.}
\label{fig:color}
\end{figure}
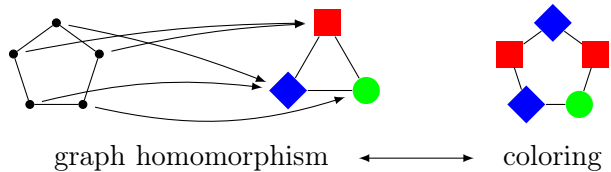

\begin{figure}
\centering
\begin{tikzpicture}[scale=.6, baseline=.5cm]
    \draw (0,0) grid (4,3);
    \foreach \x in {0,...,4}{
        \foreach \y in {0,...,3}{
            \node[P] (v\x\y) at (\x,\y) {};
        }
    }
    \foreach \c in {01,02,12,42}{
        \node[c1] at (v\c) {};
    }
    \foreach \c in {21,23,33,30,40}{
        \node[c2] at (v\c) {};
    }
\end{tikzpicture}
\caption[A configuration for the Widom--Rowlinson model on a grid.]{A configuration for the Widom--Rowlinson model on a grid, corresponding to a homomorphism to $\tikzHwr$, where vertices of the grid that are mapped to the first vertex in $\tikzHwr$ are marked \tikz[baseline=-3pt]{\node[c1] {};} and those mapped to the third vertex are marked \tikz[baseline=-3pt]{\node[c2] {};}.}
\label{fig:WR}
\end{figure}

\begin{example}[Independent sets] \label{ex:ind}
Homomorphisms from $G$ to $\tikzHind$ correspond bijectively to independent sets in $G$. Indeed, a map of vertices from $G$ to $\tikzHind$ is a homomorphism if and only if the preimage of the non-looped vertex in $\tikzHind$ forms an independent set in $G$. So $\hom(G, \tikzHind) = i(G)$. See Figure~\ref{fig:ind}. In the statistical physics literature\footnote{See \cite{BW99} for the connection between the combinatorics of graph homomorphisms and Gibbs measures in statistical physics.
}, independent sets correspond to \emph{hard-core models}. For example, they can be used to represent configurations of non-overlapping spheres (``hard'' spheres) on a grid.
\end{example}

\begin{example}[Graph colorings] \label{ex:color}
When the target graph is the complete graph $K_q$ on $q$ vertices, a graph homomorphism from $G$ to $K_q$ corresponds to a coloring of the vertices of $G$ with $q$ colors so that no two adjacent vertices of $G$ receive the same color. Such colorings are called \emph{proper $q$-colorings}. See Figure~\ref{fig:color}. Thus $\hom(G,K_q)$ is the number of proper $q$-colorings of $G$. For a fixed $G$, the quantity $\hom(G,K_q)$ is a polynomial function in $q$, and it is called the \emph{chromatic polynomial} of $G$, a classic object in graph theory.
\end{example}

\begin{example}[Widom--Rowlinson model] \label{ex:WR}
A homomorphism from $G$ to $\tikzHwr$ corresponds to a partial coloring of the vertices of $G$ with red or blue, allowing vertices to be left uncolored, such that no red vertex is adjacent to a blue vertex. Such a coloring is known as a \emph{Widom--Rowlinson configuation}. See Figure~\ref{fig:WR}.
\end{example}

As graph homomorphisms generalize independent sets, one may wonder whether theorems in Section~\ref{sec:ind} generalize to graph homomorphisms. There have indeed been some interesting results in this direction, as well as several intriguing open problems.

It turns out, perhaps surprisingly, that Theorem~\ref{thm:kahn}, concerning the number of independent sets in a regular bipartite graph, always extends to graph homomorphisms.

\begin{theorem}[Galvin and Tetali~\cite{GT04}] \label{thm:GT}
  Let $G$ be a bipartite $d$-regular graph and $H$ a loop-graph. Then
  \[
    \hom(G,H)^{1/v(G)} \le \hom(K_{d,d}, H)^{1/(2d)}.
  \]
\end{theorem}

Can the bipartite hypothesis above be dropped as in Theorem~\ref{thm:zhao}? The answer is no. Indeed, with $H=\tikzHtwoloops$ being two disjoint loops, $\hom(G,\tikzHtwoloops) = 2^{c(G)}$ where $c(G)$ is the number of connected components of $G$. In this case, $\hom(G,\tikzHtwoloops)^{1/v(G)}$ is maximized when the sizes of the components of $G$ are as small as possible (among $d$-regular graphs), i.e., when $G = K_{d+1}$.

The central problem of interest for the rest of this article is stated below. It has been solved for certain targets $H$, but it is open in general. The analogous minimization problem is also interesting, and will be discussed in Section~\ref{sec:min}.

\begin{problem} \label{prb:hom-max}
Fix a loop-graph $H$ and a positive integer $d$.
Determine the supremum of $\hom(G,H)^{1/v(G)}$ taken over all $d$-regular graphs $G$.
\end{problem}

We have already seen two cases where Problem~\ref{prb:hom-max} has been solved: when $H = \tikzHind$, the maximum is attained by $G = K_{d,d}$ (Theorem~\ref{thm:zhao}), and when $H = \tikzHtwoloops$, the maximum is attained by $G = K_{d+1}$. The latter example can be extended to $H$ being a disjoint union of complete loop-graphs. Another easy case is $H$ bipartite, as $\hom(G,H) = 0$ unless $G$ is bipartite, so the maximizer is $K_{d,d}$ by Theorem~\ref{thm:GT}.

I extended Theorem~\ref{thm:zhao} to solve Problem~\ref{prb:hom-max} for a certain family of $H$. We define a \emph{loop-threshold graph} to be a loop-graph whose vertices can be ordered so that its adjacency matrix has the property that whenever an entry is $1$, all entries to the left of it and above it are $1$ as well. An example of a loop-threshold graph, along with its adjacency matrix, is shown below.
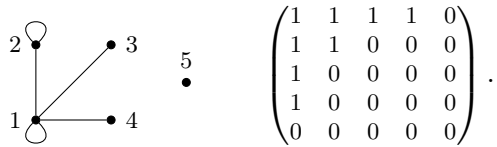
\[
\begin{tikzpicture}[baseline=(current bounding box.center),font=\footnotesize]
  \node[P,label=left:1] (1) at (0,0) {};
  \node[P,label=left:2] (2) at (0,1) {};
  \node[P,label=right:3] (3) at (1,1) {};
  \node[P,label=right:4] (4) at (1,0) {};
  \node[P,label=above:5] (5) at (2,0.5) {};
  \draw (1)--(2) (1)--(3) (1)--(4);
  \draw (1) edge[-,in=-45,out=-135,loop] ();
  \draw (2) edge[-,in=45,out=135,loop] ();
\end{tikzpicture}
\qquad
\footnotesize
\begin{pmatrix}
1&1&1&1&0\\
1&1&0&0&0\\
1&0&0&0&0\\
1&0&0&0&0\\
0&0&0&0&0
\end{pmatrix}.
\]
Loop-threshold graphs generalize \tikzHind from Example~\ref{ex:ind}. 
The following result was obtained by extending the proof method of Theorem~\ref{thm:zhao}. It answers Problem~\ref{prb:hom-max} when the target $H$ is a loop-threshold graph. 

\begin{theorem}[\cite{Zhao11}] \label{thm:threshold}
Let $G$ be a $d$-regular graph $G$ and $H$ a loop-threshold graph. Then
\[
\hom(G,H)^{1/v(G)} \le \hom(K_{d,d},H)^{1/(2d)}.
\]
\end{theorem}

In fact, the theorem was proved in \cite{Zhao11} for any $H$ that is a \emph{bipartite swapping target}, a class of loop-graphs that includes the loop-threshold graphs (see Section~\ref{sec:swap}).
Sernau~\cite{Ser} recently extended Theorem~\ref{thm:threshold} to an even larger family of $H$ (see Section~\ref{sec:prod}).

The most interesting open case of Problem \ref{prb:hom-max} is $H = K_q$, concerning the number of proper $q$-colorings of vertices of $G$ (Example~\ref{ex:color}). 

\begin{conjecture} \label{conj:coloring}
For every $d$-regular graph $G$ and integer $q \ge 3$,
\[
\hom(G,K_q)^{1/v(G)} \le \hom(K_{d,d},K_q)^{1/(2d)}.
\]
\end{conjecture}

The conjecture was recently solved for $d=3$ by Davies, Jenssen, Perkins, and Roberts \cite{DJPR3} using a novel method they developed earlier. We will discuss the method in Section~\ref{sec:occup}. The conjecture remains open for all $d\ge 4$ and $q \ge 3$. 
The above inequality is known to hold if $q$ is sufficiently large as a function of $G$ \cite{Zhao11} (the current best bound is $q > 2\binom{v(G)d/2}{4}$ \cite{Gal13}). 

The first non-trivial case of Problem~\ref{prb:hom-max} where the maximizing $G$ is not $K_{d,d}$ was obtained recently by Cohen, Perkins, and Tetali \cite{CPT}. 

\begin{theorem}[Cohen, Perkins, and Tetali~\cite{CPT}] \label{thm:wr}
For any $d$-regular graph $G$ we have
\[
\hom(G, \tikzHwr)^{1/v(G)} \le \hom(K_{d+1}, \tikzHwr)^{1/(d+1)}.
\]
\end{theorem}

Theorem~\ref{thm:wr} was initially proved \cite{CPT} using the occupancy fraction method, which will be discussed in Section~\ref{sec:occup}. Subsequently, a much shorter proof was given in \cite{CCPT} (also see Sernau \cite{Ser}).\footnote{Sernau also tackled Theorem~\ref{thm:wr}, obtaining an approximate result in a version of \cite{Ser} that predated \cite{CPT} and \cite{CCPT}. After the appearance of \cite{CCPT}, Sernau corrected an error (identified by Cohen) in \cite{Ser}, and the corrected version turned out to include Theorem~\ref{thm:wr} as a special case.} These methods can be used to prove that $K_{d+1}$ is the maximizer for a large family of target loop-graphs $H$ (see Section~\ref{sec:prod}).

There are weighted generalizations of these problems and results. Though, for clarity, we defer discussing the weighted versions until Section~\ref{sec:occup}, where we will see that introducing weights leads to a powerful new differential method for proving the unweighted results.

\medskip

We conclude this section with some open problems. Galvin~\cite{Gal13} conjectured that in Problem~\ref{prb:hom-max}, the maximizing $G$ is always either $K_{d,d}$ or $K_{d+1}$, as with all the cases we have seen so far. However, Sernau~\cite{Ser} recently found a counterexample (a similar construction was independently found by Pat Devlin; See Section~\ref{sec:neither}). As it stands, there does not seem to be a clean conjecture concerning the solution to Problem~\ref{prb:hom-max} on determining the maximizing $G$. Sernau suggested the possibility that there is a finite list of maximizing $G$ for every $d$.

\begin{conjecture}
  For every $d \ge 3$, there exists a finite set $\cG_d$ of $d$-regular graphs such that for every loop-graph $H$ and every $d$-regular graph $G$ one has
\[
   \hom(G,H)^{1/v(G)} \le \max_{G' \in \cG_d} \hom(G', H)^{1/v(G')}.
\]
\end{conjecture}

It has been speculated that the maximizing $G$ perhaps always has between $d+1$ and $2d$ vertices (corresponding to $K_{d+1}$ and $K_{d,d}$ respectively). 
Sernau suggested the possibility that for a fixed $H$, the maximizer is always one of $K_{d,d}$ and $K_{d+1}$ as long as $d$ is large enough.

\begin{conjecture}
	Let $H$ be a fixed loop-graph. There is some $d_H$ such that for all $d \ge d_H$ and $d$-regular graph $G$,
	\[
	\hom(G,H)^{1/v(G)} \le \max\{\hom (K_{d+1},H)^{1/(d+1)}, \hom(K_{2d}, H)^{1/(2d)}\}.
	\]
\end{conjecture}

We do not know if the supremum in Problem~\ref{prb:hom-max} can always be attained.

\begin{question}
	Fix $d \ge 3$ and a loop-graph $H$. Is the supremum of $\hom(G,H)^{1/v(G)}$ over all $d$-regular graphs $G$ always attained by some $G$?
\end{question}

It could be the case that the supremum is the limit coming from a sequence of graphs $G$ of increasing size instead of a single graph $G$ on finitely many vertices. This is indeed the case if we wish to \emph{minimize} $\hom(G, \tikzHwr)^{1/v(G)}$ over $d$-regular graphs $G$. Csikv\'ari~\cite{Csi16ar} recently showed that the infimum of $\hom(G, \tikzHwr)^{1/v(G)}$ is given by a limit of $d$-regular graphs $G$ with increasing girth (i.e., $G$ locally looks like a $d$-regular tree at every vertex).

\section{Projection inequalities} \label{sec:proj}

The original proofs of Theorems~\ref{thm:kahn} and \ref{thm:GT} use beautiful entropy arguments, with a key input being Shearer's entropy inequality \cite{CGFS86}. Unfortunately we will not cover the entropy arguments as they would lead us too far astray. See Galvin's lecture notes~\cite{Gal} for a nice exposition of the entropy method for counting problems. The first non-entropy proof of these two theorems was given in \cite{LZ15} using a variant of H\"older's inequality, which we describe in this section. We begin our discussion with a classical projection inequality. See Friedgut's \textsc{Monthly} article \cite{Fri04} concerning how the projection inequalities relate to entropy.

Let $P_{xy}$ denote the projection operator from $\RR^3$ onto the $xy$-plane. Similarly define $P_{xz}$ and $P_{yz}$. Let $S$ be a body in $\RR^3$ such that each of the three projections $P_{xz}(S)$, $P_{yz}(S)$, and $P_{xz}(S)$ has area $1$. What is the maximum volume of $S$? (This is not as obvious as it may first appear. Note that we are projecting onto the 2-D coordinate planes as opposed to the 1-D axes.)

The answer is $1$, attained when $S$ is an axes-parallel cube of side-length $1$. Indeed, equivalently (by re-scaling), we have
\begin{equation}\label{eq:vol-projection}
\vol(S)^2 \le \area(P_{xy}(S)) \area(P_{xz}(S)) \area(P_{yz}(S)).
\end{equation}
Such results were first obtained by Loomis and Whitney~\cite{LW49}. More generally, for any functions $f, g, h \colon \RR^2 \to \RR$ (assuming integrability conditions)
\begin{multline} \label{eq:lw-3}
\left( \int_{\RR^3} f(x,y)g(x,z)h(y,z) \, dxdydz\right)^2
\\\le 
\left( \int_{\RR^2} f(x,y)^2 \, dxdy\right)
\left(\int_{\RR^2} g(x,z)^2 \, dxdz\right)
\left( \int_{\RR^2} h(y,z)^2 \, dydz\right).
\end{multline}
To see how \eqref{eq:lw-3} implies \eqref{eq:vol-projection}, take $f, g, h$ to be the indicator functions of the projections of $S$ onto the three coordinates planes, and observe that $1_{S}(x,y,z) \le f(x,y)g(x,z)h(y,z)$.

Let us prove \eqref{eq:lw-3}. In fact, $x, y, z$ can vary over any measurable space instead of $\RR$. In our application the domains will be discrete, i.e., the integral will be a sum. It suffices to prove the inequality when $f, g, h$ are nonnegative. The proof is via three simple applications of the Cauchy--Schwarz inequality, to the variables $x$, $y$, $z$, one at a time in that order:
\begin{align*}
&\int f(x,y)g(x,z)h(y,z) \, dxdydz
\\
&\le 
\int \left(\int f(x,y)^2 \, dx\right)^{1/2} \left(\int g(x,z)^2 \, dx\right)^{1/2} h(y,z) \, dydz
\\
&\le 
\int \left(\int f(x,y)^2 \, dx dy\right)^{1/2} \left(\int g(x,z)^2 \, dx\right)^{1/2} \left(\int h(y,z)^2 \, dy\right)^{1/2} \, dz
\\
&\le 
\left(\int f(x,y)^2 \, dx dy\right)^{1/2} \left(\int g(x,z)^2 \, dx dz\right)^{1/2} \left(\int h(y,z)^2 \, dydz\right)^{1/2}
\\
&= \|f\|_2 \|g\|_2 \|h\|_2,
\end{align*}
where 
\[
\|f\|_p := \left( \int |f|^p \right)^{1/p}
\]
is the $L^p$ norm. This proves \eqref{eq:lw-3}. This inequality strengthens H\"older's inequality, since a direct application of H\"older's inequality would yield
\begin{equation}\label{eq:holder-3}
\int f g h \le \|f\|_{3}\|g\|_{3}\|h\|_{3}.
\end{equation}
What we have shown is that whenever each of the variables $x, y, z$ appears in the argument of exactly two of the three functions $f, g, h$, then the $L^3$ norms on the right-hand side of~\eqref{eq:holder-3} can be sharpened to $L^2$ norms (we always have $\| f\|_2 \le \|f\|_3$ by convexity).

The above proof easily generalizes to prove the following more general result \cite{Fin92} (also see \cite[Theorem~3.1]{LZ15}). It is also related to the Brascamp--Lieb inequality~\cite{BL76}.

\begin{theorem} \label{thm:holder-ext}
Let $A_1, \dots, A_m$ be subsets of $[n]:= \{1, 2, \dots, n\}$ such that each $i \in [n]$ appears in exactly $d$ of the sets $A_j$. Let $\Omega_i$ be a measure space for each $i \in [n]$. For each $j$, let $f_j \colon \prod_{i \in A_j} \Omega_i \to \RR$ be measurable functions. Let $P_j$ denote the projection of $\RR^n$ onto the coordinates indexed by $A_j$. Then
\[
\int_{\Omega_1 \times \cdots \times \Omega_n} f_1(P_1(\bx)) \cdots f_m(P_m(\bx)) \, d\bx \le \|f_1\|_{d} \cdots \|f_m\|_{d}.
\]
\end{theorem}

Using this inequality, we now prove Theorem~\ref{thm:GT}.

\begin{proof}[Proof of Theorem~\ref{thm:GT}] \cite{LZ15}
Let $V(G) = U \cup W$ be a bipartition of $G$. Since $G$ is $d$-regular, $|U| = |W| = v(G)/2$. For any $z_1, \dots, z_d \in V(H)$, let
\[
f(z_1, \dots, z_d) := |\{z \in V(H) : z_1z, \dots, z_dz \in E(H)\}|
\]
denote the size of the common neighborhood of $z_1, \dots, z_d$ in $H$.

For any $\phi \colon U \to V(H)$, the number of ways to extend $\phi$ to a graph homomorphism from $G$ to $H$ can be determined by noting that for each $w \in W$, there are exactly $f(\phi(u) : u \in N(w))$ choices for its image $\phi(w)$, independently of the choices for other vertices in $W$. Therefore,
\[
\hom(G, H) = \sum_{\phi \colon U \to V(H)} \prod_{w \in W} f(\phi(u) : u \in N(w)).
\]
Since $G$ is $d$-regular, every $u \in U$ is contained in $N(w)$ for exactly $d$ different $w \in W$. Therefore, by applying Theorem~\ref{thm:holder-ext} with the counting measure on $V(H)$, we find that
\[
\hom(G, H) 
\le \|f \|_{d}^{|W|}.
\]
Note that
\[
\|f\|_{d}^d = \sum_{z_1, \dots, z_d \in V(H)} f(z_1,\dots, z_d)^d = \hom(K_{d,d},H).
\]
Therefore,
\[
\hom(G, H) \le \hom(K_{d,d},H)^{|W|/d} = \hom(K_{d,d},H)^{v(G)/(2d)}. \qedhere
\]
\end{proof}

\section{A bipartite swapping trick} \label{sec:swap}

In the previous section, we proved Theorem~\ref{thm:kahn} about the maximum number of independent sets in a bipartite $d$-regular graph $G$. Now we use it to deduce Theorem~\ref{thm:zhao}, showing that the bipartite hypothesis can be dropped. The proof follows \cite{Zhao10,Zhao11}. The idea is to transform $G$ into a bipartite graph, namely the \emph{bipartite double cover} $G \x K_2$, with vertex set $V(G) \x \{0,1\}$. The vertices of $G \x K_2$ are labeled $v_i$ for $v \in V(G)$ and $i \in \{0,1\}$. Its edges are $u_0v_1$ for all $uv \in E(G)$. See Figure~\ref{fig:swap}. This construction is a special case of the graph tensor product, which we define in the next section. Note that $G \x K_2$ is always a bipartite graph. The following key lemma shows that $G \x K_2$ always has at least as many independent sets as two disjoint copies of $G$.

\begin{lemma}[\cite{Zhao10}] \label{lem:indep-bip}
  Let $G$ be any graph (not necessarily regular). Then
  \[
  i(G)^2 \le i(G \x K_2).
  \]
\end{lemma}

Since $G \x K_2$ is bipartite and $d$-regular, Theorem~\ref{thm:kahn} implies
\[
i(G)^2 \le i(G \x K_2) \le (2^{d+1} - 1)^{n/d},
\]
so that Theorem~\ref{thm:zhao} follows immediately. See Figure~\ref{fig:swap} for an illustration of the following proof.

\begin{proof}[Proof of Lemma~\ref{lem:indep-bip}]
Let $2 G$ denote a disjoint union of two copies of $G$. Label its vertices by $v_i$ with $v \in V$ and $i \in \{0,1\}$ so that its edges are $u_iv_i$ with $uv \in E(G)$ and $i \in \{0,1\}$. We will give an injection $\phi \colon I(2 G) \to I(G \x K_2)$. Recall that $I(G)$ is the set of independent sets of $G$. The injection would imply $i(G)^2 = i(2G) \le i(G \x K_2)$ as desired.

Fix an arbitrary order on all subsets of $V(G)$. 
Let $S$ be an independent set of $2G$. Let
\[
E_\bad(S) := \{uv \in E(G) : u_0, v_1 \in S\}.
\]
Note that $E_\bad(S)$ is a bipartite subgraph of $G$, since each edge of $E_\bad$ has exactly one endpoint in $\{v \in V(G) : v_0 \in S\}$ but not both (or else $S$ would not be independent). Let $A$ denote the first subset (in the previously fixed ordering) of $V(G)$ such that all edges in $E_\bad(S)$ have one vertex in $A$ and the other outside $A$. Define $\phi(S)$ to be the subset of $V(G) \x \{0,1\}$ obtained by ``swapping'' the pairs in $A$, i.e., for all $v \in A$, $v_i \in \phi(S)$ if and only if $v_{1-i} \in S$ for each $i \in \{0,1\}$, and for all $v \notin A$, $v_i \in \phi(S)$ if and only if $v_i \in S$ for each $i \in \{0,1\}$. It is not hard to verify that $\phi(S)$ is an independent set in $G \x K_2$. The swapping procedure fixes the ``bad'' edges.

It remains to verify that $\phi$ is an injection. For every $S \in I(2G)$, once we know $T = \phi(S)$, we can recover $S$ by first setting
\[
E'_\bad(T) = \{uv \in E(G) : u_i, v_i \in T \text{ for some } i \in \{0,1\} \},
\]
so that $E_\bad(S) = E_\bad'(T)$, and then finding $A$ as earlier and swapping the pairs of $A$ back. (Remark: it follows that $T \in I(G\x K_2)$ lies in the image of $\phi$ if and only if $E_\bad'(T)$ is bipartite.)
\end{proof}

\begin{figure}
\begin{tikzpicture}[scale=.8]
\begin{scope}
  \begin{scope}[shift={(-.1,.2)}]
    \node[W] (a) at (0,0) {};
    \node[B] (b) at (3,0) {};
    \node[W] (c) at (1,1) {};
    \node[W] (d) at (2,1) {};
    \node[B] (e) at (0,2) {};
    \node[W] (f) at (3,2) {};
  \end{scope}
  \begin{scope}[shift={(.1,-.2)}]
    \node[W] (a1) at (0,0) {};
    \node[W] (b1) at (3,0) {};
    \node[B] (c1) at (1,1) {};
    \node[W] (d1) at (2,1) {};
    \node[W] (e1) at (0,2) {};
    \node[B] (f1) at (3,2) {};
  \end{scope}	
  \draw (a)--(b)--(d)--(c)--(a)--(e)--(c); \draw (e)--(f)--(b); \draw (d)--(f);	
  \draw (a1)--(b1)--(d1)--(c1)--(a1)--(e1)--(c1); \draw (e1)--(f1)--(b1); \draw (d1)--(f1);
  \node at (1.5,-.8) {$2G$};
\end{scope}
\begin{scope}[shift={(6,0)}]
  \begin{scope}[shift={(-.1,.2)}]
    \node[W] (a) at (0,0) {};
    \node[B] (b) at (3,0) {};
    \node[W] (c) at (1,1) {};
    \node[W] (d) at (2,1) {};
    \node[B] (e) at (0,2) {};
    \node[W] (f) at (3,2) {};
  \end{scope}
  \begin{scope}[shift={(.1,-.2)}]
    \node[W] (a1) at (0,0) {};
    \node[W] (b1) at (3,0) {};
    \node[B] (c1) at (1,1) {};
    \node[W] (d1) at (2,1) {};
    \node[W] (e1) at (0,2) {};
    \node[B] (f1) at (3,2) {};
  \end{scope}
  \draw[dashed] (0,2) circle (.8);
  \draw[dashed] (3,0) circle (.8);
  \draw (a)--(b1)--(d)--(c1)--(a)--(e1)--(c);
  \draw (e1)--(f)--(b1); 
  \draw (d1)--(f);	
  \draw (a1)--(b)--(d1)--(c)--(a1)--(e)--(c1); 
  \draw (f1)--(b); 
  \draw (d)--(f1);
  \draw[ultra thick] (e)--(c1);
  \draw[ultra thick] (b)--(f1);
  \draw[ultra thick] (e)--(f1);
  \node at (1.5,-.8) {$G \x K_2$};        
\end{scope}
\begin{scope}[shift={(12,0)}]
  \begin{scope}[shift={(-.1,.2)}]
    \node[W] (a) at (0,0) {};
    \node[W] (b) at (3,0) {};
    \node[W] (c) at (1,1) {};
    \node[W] (d) at (2,1) {};
    \node[W] (e) at (0,2) {};
    \node[W] (f) at (3,2) {};
  \end{scope}
  \begin{scope}[shift={(.1,-.2)}]
    \node[W] (a1) at (0,0) {};
    \node[B] (b1) at (3,0) {};
    \node[B] (c1) at (1,1) {};
    \node[W] (d1) at (2,1) {};
    \node[B] (e1) at (0,2) {};
    \node[B] (f1) at (3,2) {};
  \end{scope}
  \draw (a)--(b1)--(d)--(c1)--(a)--(e1)--(c); \draw (e1)--(f)--(b1); \draw (d1)--(f);	
  \draw (a1)--(b)--(d1)--(c)--(a1)--(e)--(c1); \draw (e)--(f1)--(b); \draw (d)--(f1);		
  \node at (1.5,-.8) {$G \x K_2$};
\end{scope}
\end{tikzpicture}
\caption{The bipartite swapping trick in the proof of Lemma~\ref{lem:indep-bip}: swapping the circled pairs of vertices (denoted $A$ in the proof) fixes the bad edges (bolded), transforming an independent set of $2G$ into an independent set of $G \x K_2$.}
\label{fig:swap}
\end{figure}
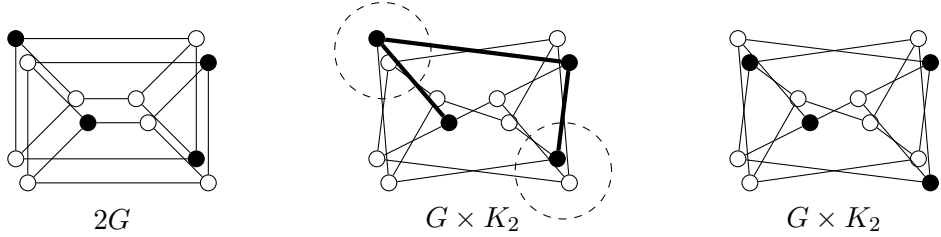

In \cite{Zhao11}, the above method was used to extend Theorem~\ref{thm:GT} to Theorem~\ref{thm:threshold}, and  more generally, a wider family of target graphs defined below.

\begin{definition}
  A loop-graph $H$ is a \emph{bipartite swapping target} if $H^\bst$ is bipartite, where $H^\bst$ is the auxiliary graph defined by taking $V(H^\bst) = V(H) \x V(H)$ and an edge between $(u,v)$ and $(u',v')$ if and only if
\[
  uu',vv' \in E(H) \quad \text{and} \quad \{uv' \notin E(H) \text{ or } u'v \notin E(H)\}.
\]
\end{definition}

\begin{theorem}[\cite{Zhao11}] \label{thm:bst} Let $H$ be a bipartite swapping target. Then $\hom(G, H)^2 \le \hom(G \x K_2, H)$ for all graphs $G$. Consequently, for $d$-regular graphs $G$, one has
\[
\hom(G, H)^{1/v(G)} \le \hom(K_{d,d},H)^{1/(2d)}.
\]
\end{theorem}

Sernau~\cite{Ser} extended the class of $H$ for which Theorem~\ref{thm:bst} holds by observing that this class is closed under taking tensor products. See Section~\ref{sec:closure-tensor}.

\begin{example} Let $H$ be $P_k$ (the path with $k$ vertices) with a single loop added at the $i$-th vertex. Then $H$ is a bipartite swapping target if $i \in \{1,2,k-1,k-2\}$. The bipartition of $H^\bst$ is indicated below by vertex colors.
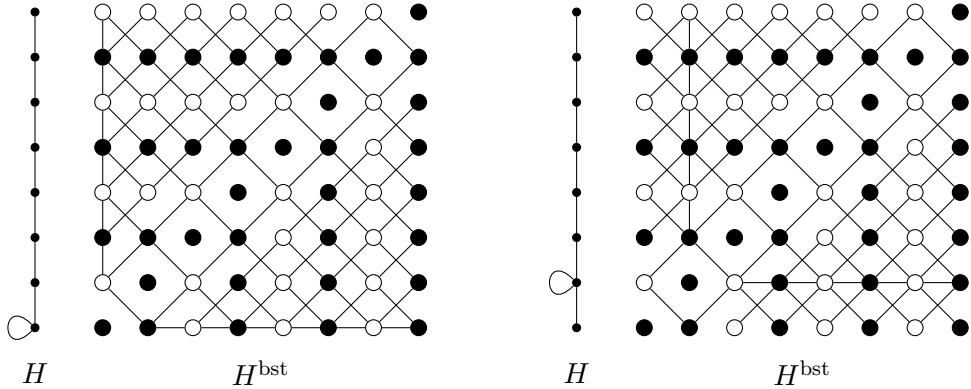
\begin{center}
\input{fig-bst-loop-path}

\end{center}
On the other hand, $P_5$ with a loop added to the middle vertex is not a bipartite swapping target, as seen by the odd cycle highlighted below.
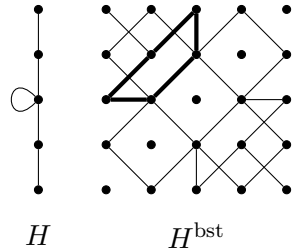
\begin{center}
\input{fig-non-bst}

\end{center}
Any $P_k$ with single loop added at vertex $i \notin \{1,2,k-1,k-2\}$ cannot be a bipartite swapping target as it contains the above $H$ as an induced subgraph.
\end{example}

The above method does not extend to $H= K_q$, corresponding to the number of proper $q$-colorings. Nonetheless the analogous strengthening of Conjecture~\ref{conj:coloring} is conjectured to hold.

\begin{conjecture}[\cite{Zhao11}] \label{conj:color-bip}
For every graph $G$ and every $q \ge 3$,
\[
\hom(G, K_q)^2 \le \hom(G \x K_2, K_q).
\]
\end{conjecture}

Conjecture~\ref{conj:color-bip} implies Conjecture~\ref{conj:coloring}. It is known \cite{Zhao11} that Conjecture~\ref{conj:color-bip} holds when $q$ is sufficiently large as a function of $G$.

\section{Graph products and powers} \label{sec:prod}

We define several operations on (loop-)graphs.
\begin{itemize}
\item \emph{Tensor product} $G \x H$: its vertices are $V(G) \x V(H)$, with $(u,v)$ and $(u',v') \in V(G) \x V(H)$ adjacent in $G \x H$ if $uu' \in E(G)$ and $vv' \in E(H)$. This construction is also known as the \emph{categorical product}.
\item \emph{Exponentiation} $H^G$: its vertices are maps $f \colon V(G) \to V(H)$ (not necessarily homomorphisms), where $f$ and $f'$ are adjacent if $f(u)f'(v) \in E(H)$ whenever $uv \in E(G)$.
\item $G^\circ$: same as $G$ except that every vertex now has a loop.
\item $\ell(H)$: subgraph of $H$ induced by its looped vertices, or equivalently, delete all non-looped vertices from $H$.
\end{itemize}

\begin{example}
	The tensor product $G \x K_2$ (here $K_2 = \tikzKtwo$) is the bipartite double cover used in the previous section (see Figure~\ref{fig:swap}).
\end{example}

\begin{example}
	We have $\tikzHind \x \tikzHind = \tikzHindsquared$, with adjacency matrix 
	$\left(\begin{smallmatrix}
	1 & 1 & 1 & 1 \\
	1 & 0 & 1 & 0 \\
	1 & 1 & 0 & 0 \\
	1 & 0 & 0 & 0		
	\end{smallmatrix}\right)$,                  
	which can be obtained by taking the adjacency matrix      $\left(\begin{smallmatrix}
	1 & 1 \\
	1 & 0
	\end{smallmatrix}\right)$
	of $\tikzHind$ and then replacing each $1$ in the matrix by a copy of 
	$\left(\begin{smallmatrix}
	1 & 1 \\
	1 & 0
	\end{smallmatrix}\right)$
	and replacing each $0$ by a copy of 
	$\left(\begin{smallmatrix}
	0 & 0 \\
	0 & 0
	\end{smallmatrix}\right)$. More generally, the adjacency matrix of a tensor product of (loop-)graphs is the matrix tensor product of the adjacency matrices of the graphs.
\end{example}

\begin{example}\label{ex:wr-ind-power-loop}
For any loop-graph $H$, the graph $H^{K_2}$ has vertex set $V(H) \x V(H)$, with $(u,v)$ and  $(u',v') \in V(H) \x V(H)$ adjacent if and only if $uv', u'v \in E(H)$. In particular, if $H_\ind = \tikzHind$, then $H_\ind^{K_2} = \tikzHwrdown$.
\end{example}

Here are a few easy yet key facts relating the above operations with graph homomorphisms. The proofs are left as exercises for the readers.
\begin{equation}
  \label{eq:H-prod}
  \hom(G, H_1 \x H_2) = \hom(G, H_1) \hom(G, H_2)
\end{equation}
\begin{equation}
  \label{eq:H-power}
  \hom(G \x G', H) = \hom(G, H^{G'})
\end{equation}
\begin{equation}
  \label{eq:G-loop}
  \hom(G^\circ, H) = \hom(G, \ell(H)).
\end{equation}

\subsection{$K_{d+1}$ as maximizer}

Now we prove Theorem~\ref{thm:wr} concerning the Widom--Rowlinson model. Recall it says that for any $d$-regular graph $G$, we have
\[
\hom(G, \tikzHwr)^{1/v(G)} \le \hom(K_{d+1}, \tikzHwr)^{1/(d+1)}.
\]

\begin{proof}[Proof of Theorem~\ref{thm:wr}] 
  \cite{CCPT}
  We have $\tikzHwr = \ell(H_\ind^{K_2})$ (Example~\ref{ex:wr-ind-power-loop}). For any graph $G$,
\begin{equation} \label{eq:wr-ind}
\hom(G,\tikzHwr) = \hom(G,\ell(H_\ind^{K_2})) = \hom(G^\circ, H_\ind^{K_2}) = \hom(G^\circ \x K_2, \tikzHind).
\end{equation}
When $G$ is $d$-regular, $G^\circ \x K_2$ is a $(d+1)$-regular bipartite graph, so Theorem~\ref{thm:GT} (or Theorem~\ref{thm:kahn}) implies that the above quantity is at most $\hom(K_{d+1,d+1},\tikzHind)^{v(G)/(d+1)}$. Since $K_{d+1,d+1} = K_{d+1}^\circ \x K_2$, we have by~\eqref{eq:wr-ind},
\begin{align*}
\hom(G,\tikzHwr)^{1/v(G)} 
&= \hom(G^\circ \x K_2, \tikzHind)^{1/v(G)}
\\
&\le \hom(K_{d+1}^\circ \x K_2,\tikzHind)^{1/(d+1)}
\\
&
= \hom(K_{d+1}, \tikzHwr)^{1/(d+1)}. \qedhere
\end{align*}
\end{proof}

The above proof exploits the connection~\eqref{eq:wr-ind} between the hard-core model (independent sets) and the Widom--Rowlinson model. This relationship had been previously observed in \cite{BHW99}.
More generally, the above proof extends to give the following result.

\begin{theorem}[Sernau \cite{Ser}] \label{thm:Serneau-loop-power}
Let $H = \ell(A^B)$ where $A$ is any loop-graph and $B$ is a bipartite graph. For any $d$-regular graph $G$, 
\[
\hom(G, H)^{1/v(G)} \le \hom(K_{d+1},H)^{1/(d+1)}.
\]
\end{theorem}

\begin{proof}
For every bipartite graph $B$, any product of the form $G \x B$ is bipartite, and furthermore $B \x K_2$ is two disjoint copies of $B$. For $A$ and $B$ as in the hypothesis of the theorem, we have, for any graph $G$,
\begin{align*}
\hom(G, \ell(A^B)) 
&= \hom(G^\circ, A^B)
= \hom(G^\circ \x B, A)
\\
&= \hom(G^\circ \x B \x K_2, A)^{1/2}
= \hom(G^\circ \x K_2, A^B)^{1/2}.
\end{align*}
Since $G$ is $d$-regular, $G^\circ \x K_2$ is a $(d+1)$-regular bipartite graph, so Theorem~\ref{thm:GT} implies that (recall $K_{d+1}^\circ \x K_2 \cong K_{d+1,d+1}$)
\begin{align*}
\hom(G, \ell(A^B))^{1/v(G)} 
&= \hom(G^\circ \x K_2, A^B)^{1/(2v(G))}
\\
&\le \hom(K_{d+1}^\circ \x K_2, A^B)^{1/(2d+2)}
\\
&= \hom(K_{d+1},\ell(A^B))^{1/(d+1)}. \qedhere
\end{align*}
\end{proof}

One can extend Theorem~\ref{thm:Serneau-loop-power} by considering a bigraph version of Theorem~\ref{thm:GT}. A \emph{bigraph} $G$ is a bipartite graph $G$ along with a specified left/right-vertex bipartition $V(G) = V_L(G) \cup V_R(G)$. Given two bigraphs $G$ and $H$, a  homomorphism $\phi$ from $G$ to $H$ is a graph homomorphism of the underlying graphs that respects the vertex bipartition, i.e., $\phi(V_L(G)) \subseteq V_L(H)$ and $\phi(V_R(G)) \subseteq V_R(H)$. Theorem~\ref{thm:GT}, with essentially the same proof, holds for bigraphs $G$ and $H$. The proof of Theorem~\ref{thm:Serneau-loop-power} can then be easily modified to establish the following result.

\begin{theorem} \label{thm:H-bigraph-hom}
	Let $A$ and $B$ be two bigraphs. Let $H$ denote the loop-graph with vertices being bigraph homomorphisms from $B$ to $A$, such that $\phi, \phi' \in V(H)$ are adjacent if and only if $\phi(u)\phi'(v) \in E(A)$ whenever $uv \in E(B)$ (in particular, all vertices of $H$ are automatically looped). Then for any $d$-regular graph $G$, one has
	\[
	\hom(G,H)^{1/v(G)} \le \hom(K_{d+1}, H)^{1/(d+1)}.
	\]
\end{theorem}

It may not be obvious which $H$ can arise in Theorem~\ref{thm:H-bigraph-hom}. The following special case, proved in \cite{CCPT} (prior to \cite{Ser}), provides some a nice family of examples.

\begin{definition} \label{def:extended-line-graph}
The \emph{extended line graph} $\wt H$ of a graph $H$ has $V(\wt H) = E(H)$ and two edges $e$ and $f$ of $H$ are adjacent in $\wt H$ if
\begin{enumerate}
\item $e = f$, or
\item $e$ and $f$ share a common vertex, or
\item $e$ and $f$ are opposite edges of a $4$-cycle in $G$.
\end{enumerate}
\end{definition}

Note that every vertex of $\wt H$ is automatically looped. If $B$ is a bipartite graph, then the graph $H$ in Theorem~\ref{thm:H-bigraph-hom} that arises from $A = K_2$ and $B$ is precisely $\wt B$.

\begin{corollary}[\cite{CCPT}] \label{cor:H-line-graph}
  Let $\wt H$ be the extended line graph of a bipartite graph $H$. For any $d$-regular graph $G$,
\[
\hom(G,\wt H)^{1/v(G)} \le \hom(K_{d+1}, \wt H)^{n/(d+1)}.
\]
\end{corollary}

For a simple graph $H$, let $H^\circ$ denote $H$ with a loop added at every vertex. Let $P_k$ denote the path of $k$ vertices, and $C_k$ the cycle with $k$ vertices.

\begin{example}
One has $\wt P_{k+1} = P_{k}^\circ$ for all $k$. Also, $\wt C_k = C_k^\circ$ for all $k \ne 4$.
\end{example}

\begin{corollary}[\cite{CCPT}]
Let $H = C_k^\circ$ with even $k \ge 6$ or $H = P_k^\circ$ for any $k \ge 1$. For any $d$-regular graph $G$,
\[
\hom(G, H)^{1/v(G)} \le \hom(K_{d+1},H)^{1/v(G)}.
\]
\end{corollary}

\subsection{Closure under tensor products} \label{sec:closure-tensor}

Sernau~\cite{Ser} observed that, for any $d$, if $H = H_1$ and $H = H_2$ both have the property that $G = K_{d,d}$ maximizes $\hom(G,H)^{1/v(G)}$ over all $d$-regular graphs $G$, then $H = H_1 \x H_2$ has the same property by \eqref{eq:H-prod}. In other words, the set of $H$ such that $G = K_{d,d}$ is the maximizer in Problem~\ref{prb:hom-max} is closed under tensor products. This observation enlarges the set of such $H$ previously obtained in Theorems~\ref{thm:threshold} and \ref{thm:bst}.

Similarly, the set of $H$ such that $G = K_{d+1}$ maximizes the expression $\hom(G,H)^{1/v(G)}$ over $d$-regular graphs $G$ is also closed under tensor products. However, the loop-graphs $H$ that arise in Theorem~\ref{thm:H-bigraph-hom} are already closed under taking tensor products, so no new cases are obtained via taking the tensor product closure.

\section{Neither $K_{d,d}$ nor $K_{d+1}$} \label{sec:neither}

So far in all cases of Problem~\ref{prb:hom-max} that we have considered, the maximizing $G$ is always either $K_{d,d}$ or $K_{d+1}$. It was conjectured~\cite{Gal13} that one of $K_{d,d}$ and $K_{d+1}$ always maximizes $\hom(G, H)^{1/v(G)}$ for every $H$. However, Sernau~\cite{Ser} showed that this is false (a similar construction was independently found by Pat Devlin).

Let $d \ge 4$ and let $G$ be a $d$-regular graph with $v(G) < 2d$ other than $K_{d+1}$. Brooks' theorem tells us that $G$ is $d$-colorable, so that $\hom(G, K_d) > 0$. It follows that for this $G$, 
\begin{align*}
\hom(G, k K_d)^{1/v(G)} 
&= k^{1/v(G)} \hom(G,K_d)^{1/v(G)} 
\\
&> k^{1/(2d)} \hom(K_{d,d},K_d)^{1/(2d)} =
 \hom(K_{d,d},k K_d)^{1/(2d)}
\end{align*}
for sufficiently large $k$ (as a function of $d$) since $v(G) < 2d$. Also,
\[
\hom(G, k K_d)^{1/v(G)} > 0 = \hom(K_{d+1},k K_d)^{1/(d+1)}.
\]
Therefore neither $G = K_{d,d}$ nor $G = K_{d+1}$ maximize $\hom(G, k K_d)^{1/v(G)}$ over all $d$-regular graphs $G$.\footnote{If we wish the target graph $H$ to be connected, we can slightly modify the construction by connecting the disjoint copies by paths of length two, say.} However, we do not know which $G$ maximizes $\hom(G,k K_d)^{1/v(G)}$. For $d=3$, Csikv\'ari \cite{Csi-p} found a counterexample using a similar construction. 

In general, we do not know which graphs $G$ (other than $K_{d+1}$ and $K_{d,d}$) can arise as maximizers for $\hom(G, H)^{1/v(G)}$ in Problem~\ref{prb:hom-max}. See the end of Section~\ref{sec:hom} for some open questions and conjectures.

\section{Occupancy fraction} \label{sec:occup}

The original proof of Theorem~\ref{thm:kahn} used the entropy method~\cite{Kahn01}. The proof in Section~\ref{sec:proj}, following~\cite{LZ15}, used a variant of the H\"older's inequality, and is related to the original entropy method proof.
Recently, an elegant new proof of the result was found \cite{DJPR1} using a novel method, unrelated to previous proofs. We discuss this new technique in this section. It will be necessary to introduce weighted versions of the problems.

The \emph{independence polynomial} of a graph $G$ is defined by
\[
P_G(\lambda) := \sum_{I \in I(G)} \lambda^{|I|}.
\]
Recall that $I(G)$ is the set of independent sets of $G$. In particular, $P_G(1) = i(G)$. Theorem~\ref{thm:kahn}, which says that $i(G)^{1/v(G)} \le i(K_{d,d})^{1/(2d)}$ for $d$-regular bipartite $G$, extends to this weighted version of the number of independent sets \cite{GT04}\footnote{The case $\lambda \ge 1$ had been established earlier by Kahn~\cite{Kahn02}.}. The bipartite swapping trick in Section~\ref{sec:swap} also extends to the weighted setting~\cite{Zhao10}. 

\begin{theorem}[\cite{GT04} for bipartite $G$; \cite{Zhao10} for general $G$] \label{thm:indep-poly}
If $G$ is a $d$-regular graph and $\lambda \ge 0$, then
\[
P_G(\lambda)^{1/v(G)} \le P_{K_{d,d}}(\lambda)^{1/(2d)}.
\]
\end{theorem}

The \emph{hard-core model} with \emph{fugacity} $\lambda$ on $G$ is defined as the probability distribution on independent sets of $G$ where an independent set $I$ is chosen with probability proportional to $\lambda^{|I|}$, i.e., with probability
\[
\Pr_\lambda[I] = \frac{\lambda^{|I|}}{P_G(\lambda)}.
\]
The \emph{occupancy fraction} of $I$ is the fraction of vertices of $G$ occupied by $I$. The expected occupancy fraction of a random independent set from the hard-core model is
\[
\alpha_G(\lambda) := \frac{1}{v(G)} \sum_{I \in I(G)} |I| \cdot \Pr_\lambda[I]
= \frac{\sum_{I \in I(G)} |I| \lambda^{|I|}}{v(G) P_G(\lambda)}
= \frac{\lambda P'_G(\lambda)}{v(G) P_G(\lambda)}.
\]
The occupancy fraction is an ``observable''---a quantity associated with each instance produced by the model.

It turns out that $K_{d,d}$  maximizes the occupancy fraction among all $d$-regular graphs.

\begin{theorem}[Davies, Jenssen, Perkins, and Roberts~\cite{DJPR1}] \label{thm:occupancy}  
  For all $d$-regular graphs $G$ and all $\lambda\ge 0$, we have
  \begin{equation}\label{eq:thm-occupancy}
  \alpha_G(\lambda) \le \alpha_{K_{d,d}}(\lambda) = \frac{\lambda(1+\lambda)^{d-1}}{2(1+\lambda)^d - 1}.
\end{equation}
\end{theorem}

Since the expected occupancy fraction is proportional to the logarithmic derivative of $P_G(\lambda)^{1/v(G)}$, the inequality for the expected occupancy fraction implies the corresponding inequality for the independence polynomial. Indeed, Theorem~\ref{thm:occupancy} implies Theorem~\ref{thm:indep-poly} (and hence Theorems~\ref{thm:kahn} and \ref{thm:zhao}) since 
\[
\frac{1}{v(G)}\int_0^\lambda \frac{\ol\alpha_G(t)}{t} \, dt
= \frac{1}{v(G)}\int_0^\lambda \frac{P'_G(t)}{P_G(t)} \, dt
= \frac{\log P_G(\lambda)}{v(G)}.
\]

We reproduce here two proofs of Theorem~\ref{thm:occupancy}. They are both based on the following idea, introduced in \cite{DJPR1} for this problem. We draw a random independent set $I$ from the hard-core model and look at the neighborhood of a uniform random vertex $v \in V(G)$. The expected occupancy fraction is then the probability that $v \in I$. (It is helpful here that the occupancy fraction is an observable quantity.) We then analyze how the neighborhood of $v$ should look in relation to $I$. Since the graph is regular, a uniform random neighbor of $v$ is uniformly distributed in $V(G)$. By finding an appropriate set of constraints on the probabilities of seeing various neighborhood configurations of $v$, we can bound the probability that $v \in I$.

The first proof is given below under the additional simplifying assumption that $G$ is triangle-free (which includes all bipartite graphs and much more). See~\cite{DJPR1} for how to extend this proof to all regular graphs.

\begin{proof}[Proof of Theorem~\ref{thm:occupancy} for triangle-free $G$]
  Let $I$ be an independent set of $G$ drawn according to the hard-core model with fugacity $\lambda$. For each $v \in V(G)$, let $p_v$ denote the probability that $v \in I$. We say $v \in V(G)$ is \emph{uncovered} if none of the neighbors of $v$ are in $I$, i.e., $N(v) \cap I = \emptyset$. If $v \in I$ then $v$ is necessarily uncovered. Conversely, conditioned on $v$ being uncovered, one has $v \in I$ with probability $\lambda/(1+\lambda)$. So the probability that $v$ is uncovered is $p_v(1+\lambda) / \lambda$.

Let $U_v$ denote the set of uncovered neighbors of $v$. Since $G$ is triangle-free, $U_v$ is an independent set. Conditioned on $U_v$ being the uncovered neighbors of $v$, the probability that $v$ is uncovered, which is equivalent to $U_v \cap I = \emptyset$, is exactly $(1+\lambda)^{-|U_v|}$. Hence
\begin{equation} \label{eq:p_v-ineq}
\frac{1+\lambda}{\lambda} p_v = \EE[(1 + \lambda)^{-|U_v|}]
\le 1 - \frac{\EE[|U_v|]}{d} \left( 1 - (1+\lambda)^{-d}\right),
\end{equation}
where the inequality follows from $0 \le |U_v| \le d$ and the convexity of the function $x \mapsto (1+\lambda)^{-x}$, so that $(1+\lambda)^{-x} \le 1 - \frac{x}{d}(1-(1+\lambda)^{-d})$ for all $0 \le x \le d$ by linear interpolation.

If $v$ is chosen from $V(G)$ uniformly at random, then $\EE[p_v] = \alpha_G(\lambda)$ is the expected occupancy fraction. Similarly, $\EE[|U_v|]/d$ is the probability that a random vertex is uncovered (here we use again that $G$ is $d$-regular), which equals $\EE[p_v] \frac{1+ \lambda}{\lambda} = \alpha_G(\lambda) \frac{1+\lambda}{\lambda}$. Setting into \eqref{eq:p_v-ineq}, we obtain
\[
\frac{1+\lambda}{\lambda} \alpha_G(\lambda) \le 1 - \alpha_G(\lambda) \frac{1+\lambda}{\lambda} \left( 1 - (1+\lambda)^{-d}\right).
\]
Rearranging gives us \eqref{eq:thm-occupancy}.
\end{proof}

In \cite{DJPR1}, Theorem~\ref{thm:occupancy} was proved for all $d$-regular graphs $G$ by considering all graphs on $d$ vertices that could be induced by the neighborhood of a vertex in $G$ and using a linear program to constrain the probability distribution of the neighborhood profile of a random vertex. When $G$ is triangle-free, the neighborhood of a vertex is always an independent set, which significantly simplifies the situation. The following conjecture extends Theorem~\ref{thm:GT} to triangle-free graphs.

\begin{conjecture}[\cite{CCPT}]
  Let $G$ be a triangle-free $d$-regular graph and $H$ a loop-graph. Then
  \[
  \hom(G,H)^{1/v(G)} \le \hom(K_{d,d},H)^{1/(2d)}.
  \]
\end{conjecture}

Next we give an alternative proof of Theorem~\ref{thm:occupancy} due to Perkins~\cite{Perkins-pc}, based on a similar idea. In the following proof, we do not need to assume that $G$ is triangle-free. In the proof, we introduce an additional constraint, which allows us to obtain the result more quickly. This simplification seems to be somewhat specific to independent sets.

\begin{proof}[Second proof of Theorem~\ref{thm:occupancy}]
Let $I$ be an independent set of $G$ drawn according to the hard-core model with fugacity $\lambda$, and let $v$ be a uniform random vertex in $G$. Let $Y = \abs{I \cap N(v)}$ denote the number of neighbors of $v$ in $I$ (not including $v$ itself). Let $p_k = \PP(Y = k)$. Since $Y \in \{0, 1, \dots, d\}$,
\begin{equation}\label{eq:occ-prob-sum}
p_0 + p_1 + \cdots + p_d = 1.
\end{equation}
However, not all vectors of probabilities $(p_0, \dots, p_d)$ are feasible. The art of the method is in finding additional constraints on the probability distribution.

As in the previous proof, since $v$ is uncovered if and only if $Y = 0$, we have
\[
\alpha_G(\lambda) = \PP(v \in I) = \frac{\lambda}{1+\lambda} \PP(Y = 0) = \frac{\lambda}{1+\lambda} p_0.
\]
On the other hand, since $G$ is $d$-regular, a uniform random neighbor of $v$ is also uniformly distributed in $V(G)$, so we have
\[
\alpha_G(\lambda) = \frac{1}{d} \EE[Y] = \frac{1}{d}(p_1 + 2p_2 + \cdots + dp_d).
\]
Comparing the previous two relations, we obtain
\begin{equation}
	\label{eq:occ-neighbor-relation}
	\frac{\lambda}{1+\lambda} p_0 = \frac{1}{d}(p_1 + 2p_2 + \cdots + dp_d).
\end{equation}
Now, let us compare the probability that $v$ has $k$ versus $k-1$ neighbors. In an event where exactly $k$ neighbors of $v$ are occupied, we can remove any of the occupied neighbors from $I$, and obtain another independent set where $v$ has exactly $k-1$ neighbors. There are $k$ ways to remove an element, but we over-count by a factor of at most $d-k+1$. Also factoring in the weight multiplier, we obtain the inequality
\begin{equation} \label{eq:occ-descend}
(d-k+1) \lambda p_{k-1} \ge k p_k, \qquad \text{for } 2 \le k \le d.
\end{equation}
The constraints \eqref{eq:occ-prob-sum}, \eqref{eq:occ-neighbor-relation}, and \eqref{eq:occ-descend} together form a linear program with variables $p_0, \dots, p_d$. 
Next we show that these linear constraints together imply $p_0 \le \frac{(1+\lambda)^d}{2(1+\lambda)^d - 1}$, which gives the desired bound on $\alpha_G(\lambda) = \frac{\lambda}{1+\lambda}p_0$. Equality is attained for the probability distribution $(p_0, \dots, p_d)$ arising from $G = K_{d,d}$.

To prove this claim, first we show that if $(p_0, \dots, p_d)$ achieves the maximum of value of $p_0$ while satisfying the constraints \eqref{eq:occ-prob-sum}, \eqref{eq:occ-neighbor-relation}, and \eqref{eq:occ-descend}, then every inequality in \eqref{eq:occ-descend} must be an equality. Indeed, if we have $(d-k+1) \lambda p_{k-1} > k p_k$ for some $k$, then by increasing $p_0$ by $\epsilon$, decreasing $p_{k-1}$ by $(\frac{d\lambda}{1+\lambda} + k)\epsilon$, increasing $p_k$ by $(\frac{d\lambda}{1+\lambda} + k-1)\epsilon$, and leaving all other $p_i$'s fixed, we can maintain all constraints and increase $p_0$, provided $\epsilon > 0$ is sufficiently small. Thus, in the maximizing solution, equality occurs in \eqref{eq:occ-descend} for all $2 \le k \le d$. It can be checked that the vector $(p_0, \dots, p_d)$ arising from $G = K_{d,d}$ satisfies all the equality constraints, and it is the unique solution since we have a linear system of equations with full rank.
\end{proof}

Conjecture~\ref{conj:coloring} about the number of colorings was recently proved~\cite{DJPR3} for 3-regular graphs using an extension of the above method. Instead of the independence polynomial and the hard-core model, one considers a continuous relaxation of proper colorings by using the Potts model. We sample a $q$-coloring of $G$, not necessarily proper, so that the coloring $\sigma$ is chosen with probability proportional to $e^{-\beta m(\sigma)}$, where $m(\sigma)$ is the number of monochromatic edges, and $\beta \in \RR$ is called the \emph{inverse temperature}. A proper coloring corresponds to $\beta \to +\infty$. The \emph{partition function} (analogous to the independence polynomial) for this Potts model is
\[
Z_G^q (\beta) = \sum_{\sigma \colon V(G) \to [q]} e^{-\beta m(\sigma)}
\]
where $\sigma$ ranges over all $q$-colorings of $G$ (not necessarily proper). In the Potts model, the coloring $\sigma$ appears with probability $e^{-\beta m(\sigma)}/Z_G^q(\beta)$. The expected number of monochromatic edges of $\sigma$ is
\begin{align*}
U_G^q(\beta) 
:= \frac{1}{v(G)} \EE_\sigma[m(\sigma)] 
&= \frac{1}{v(G) Z_G^q(\beta)} \sum_{\sigma\colon V(G) \to [q]} m(\sigma) e^{-\beta m(\sigma)}
\\
&= \frac{-1}{v(G)} \frac{d}{d\beta} (\log Z_G^q(\beta)).
\end{align*}
The above quantity is analogous to the occupancy fraction for the hard-core model (think $\lambda = e^{-\beta}$). Conjecture~\eqref{conj:coloring} would follow from the next conjecture. (Note that the inequality sign is reversed since $U_G^q(\beta)$ is proportional to the negative logarithmic derivative of $Z_G(\beta)^{1/v(G)}$)

\begin{conjecture} \label{conj:potts}
	For every $d$-regular graph $G$ and integer $q \ge 3$, and any $\beta > 0$,
	\[
	U_G^q(\beta) \ge U_{K_{d,d}}^q(\beta).
	\]
	Consequently (by integrating $\beta$ and noting $Z_G^q(0)^{1/v(G)} = q$),
	\[
	Z_G^q(\beta)^{1/v(G)} \le Z_{K_{d,d}}^q(\beta)^{1/(2d)}.
	\]
\end{conjecture}

Conjecture~\ref{conj:potts} was proved for 3-regular graphs in \cite{DJPR3} using a variant of the method discussed in this section, by considering all configurations of the 2-step neighborhood of a uniform random vertex. The analysis is substantially more involved than the proofs we saw for independent sets.

\section{On the minimum number of independent sets and homomorphisms} \label{sec:min}

\subsection{Independent sets}
Having explored the maximum number of independent sets in a regular graph, let us turn to the natural opposite question. Which $d$-regular graph has the minimum number of independent sets? It turns out that the answer is a disjoint union of cliques.

\begin{theorem}[Cutler and Radcliffe~\cite{CR14}] \label{thm:ind-min}
  For a $d$-regular graph $G$,
\[
i(G)^{1/v(G)} \ge i(K_{d+1})^{1/(d+1)} = (d+2)^{1/(d+1)}.
\]
\end{theorem}

In fact, a stronger result holds: a disjoint union of $K_{d+1}$'s minimizes the number of independent sets of every fixed size. We write $a G$ for a disjoint union $a$ copies of $G$. Let $i_t(G)$ denote the number of independent sets of $G$ of size $t$.

\begin{theorem}[\cite{CR14}] \label{thm:ind-min-by-size}
  Let $a$ and $d$ be positive integers. Let $G$ be a $d$-regular graph with $a(d+1)$ vertices. Then $i_t(G) \ge i_t(a K_{d+1})$ for every $0 \le t \le a(d+1)$.
\end{theorem}

\begin{proof}
  Let us compare the number of sequences of $t$ vertices that form an independent set in $G$ and $a K_{d+1}$. In $a K_{d+1}$, we have $a(d+1)$ choices for the first vertex. Once the first vertex has been chosen, there are exactly $(a-1)(d+1)$ choices for the second vertex. More generally, for $1 \le j \le a$, once the first $j-1$ vertices have been chosen, there are exactly $(a+1-j)(d+1)$ choices for the $j$-th vertex.

On the other hand, in $G$, after the first $j-1$ vertices have been chosen, the union of these $j-1$ vertices along with their neighborhoods has cardinality at most $(j-1)(d+1)$, so there are at least $(a + 1 - j)(d+1)$ choices for the $j$-th vertex, at least as many compared to $a K_{d+1}$.
\end{proof}

\begin{proof}[Proof of Theorem~\ref{thm:ind-min}]
  Theorem~\ref{thm:ind-min-by-size} implies that $i(G)^{1/v(G)} \ge i(K_{d+1})^{1/(d+1)}$ whenever $v(G)$ is divisible by $d+1$. When $v(G)$ is not divisible by $d+1$, we can apply the same inequality to a disjoint union of $(d+1)$ copies of $G$ to obtain $i(G)^{1/v(G)} = i((d+1)G)^{1/((d+1)v(G))} \ge i(K_{d+1})^{1/(d+1)}$.
\end{proof}

The situation changes significantly if we require $G$ to be bipartite. In this case, the problem was solved very recently by Csikv\'ari~\cite{Csi16ar}, who showed that the infimum of $i(G)^{1/v(G)}$ over $d$-regular bipartite graphs $G$ is obtained by taking a sequence of $G$ with increasing girth, i.e., $G$ is locally tree-like. The limit of $i(G)^{1/v(G)}$ for a sequence of bipartite $d$-regular graphs $G$ of increasing girth was determined by Sly and Sun \cite{SS14} using sophisticated (rigorous) methods from statistical physics.

\subsection{Colorings}

Here is an infimum of $\hom(G,K_q)^{1/v(G)}$ over $d$-regular graphs $G$ due to Csikv\'ari~\cite{Csi-p}.

\begin{theorem} \label{thm:min-color}
For a $d$-regular graph $G$	and any $q \ge 2$, 
\[
\hom(G, K_q)^{1/v(G)} \ge \hom(K_{d+1}, K_q)^{1/(d+1)}.
\]
\end{theorem}

\begin{proof}
Assume $q \ge d+1$, since otherwise the right-hand side is zero. Let $\sigma$ be a random permutation of $V(G)$. For each $u \in V(G)$, let $d_u^\sigma$ denote the number of neighbors of $u$ that appears before $u$ in the permutation $\sigma$.  By coloring the vertices in the order of $\sigma$, there are at least $q - d_u^\sigma$ choices for the color of vertex $u$, so
\[
\hom(G, K_q) \ge \prod_{u \in V(G)} (q - d_u^\sigma).
\]
Taking the logarithm of both sides, we find that
\begin{equation} \label{eq:color-sigma-ineq}
\frac{1}{v(G)} \log \hom(G, K_q) \ge \frac{1}{v(G)} \sum_{u \in V(G)} \log(q - d_u^\sigma).
\end{equation}
For each $u \in V(G)$, the random variable $d_u^\sigma$ is uniformly distributed among $\{0, 1, \dots, d\}$ since the ordering of $u \cup N(u)$ under $\sigma$ is uniform. Therefore, the expected value of the right-hand side of \eqref{eq:color-sigma-ineq} is
\[
\frac{1}{d+1}(\log q + \log(q-1) + \cdots + \log(q-d)) = \frac{1}{d+1}\log \hom(K_{d+1},K_q),
\]
which proves the theorem.
\end{proof}

What is infimum of $\hom(G, K_q)^{1/v(G)}$ over \emph{bipartite} $d$-regular graphs $G$? The following inequality was proved by Csikv\'ari and Lin~\cite{CL}. For $q \ge d+1$, the constant in the inequality is best possible as it is the limit for any sequence of $d$-regular graphs with increasing girth \cite{BG08}.

\begin{theorem}[\cite{CL}] \label{thm:min-color-bip}
For any $d$-regular bipartite graph $G$ and any $q \ge 2$,
\[
\hom(G,K_q)^{1/v(G)} \ge q (1-1/q)^{d/2}.
\]	
\end{theorem}

\subsection{Widom--Rowlinson model} In the previous two cases, for independent sets and colorings, the minimizing $G$ is $K_{d+1}$, and if we restrict to bipartite $G$, the ``minimizing'' $G$ is locally tree-like. For the Widom--Rowlinson model, we saw in Theorem~\ref{thm:wr} that the quantity $\hom(G, \tikzHwr)^{1/v(G)}$  is maximized, over $d$-regular graphs $G$, by $G = K_{d+1}$. Csikv\'ari~\cite{Csi16ar} recently showed that $\hom(G, \tikzHwr)^{1/v(G)}$ is minimized, over $d$-regular graphs $G$, by a sequence of graphs $G$ with increasing girth, even without the bipartite assumption on $G$.

\section{Related results and further questions} \label{sec:related}

In this final section we mention some related results and problems. Also see the survey~\cite{Cut12} for a related discussion.

\subsection{Independent sets of fixed size}

We saw in Theorems~\ref{thm:kahn} and \ref{thm:zhao} that in the family of $d$-regular graphs on $n$ vertices, a disjoint union of $K_{d,d}$'s maximizes the number of independent sets. It is conjectured that latter maximizes the number of independents of every fixed size. Let $i_t(G)$ denote the number of independent sets of size $t$ in $G$. Recall that $k G$ denotes a disjoint union of $k$ copies of $G$.

\begin{conjecture}[\cite{Kahn01}] \label{conj:ind-fixed-size}
  If $G$ is a $d$-regular graph with $2ad$ vertices, then $i_t(G) \le i_t(a K_{d,d})$ for every $t$.
\end{conjecture}

See \cite[Section 8]{DJPR1} for the current best bounds on this problem.

\subsection{Homomorphism with weights}

Theorem~\ref{thm:GT} holds more generally for weighted graph homomorphisms, allowing $H$ to have weights on its vertices and edges. The proof in Section~\ref{sec:prod} also extends to the weighted setting after small modifications. We refer to \cite{GT04} and \cite{LZ15} for details.

\subsection{Biregular graphs}

An $(a,b)$-biregular graph is a bipartite graph such that all vertices on one side of the bipartition have degree $a$, and all vertices on the other side have degree $b$. Theorems~\ref{thm:kahn} and \ref{thm:GT} extend to biregular graphs, stating that for any $(a,b)$-biregular graph $G$ and loop-graph $H$,
\[
\hom(G,H)^{1/v(G)} \le \hom(K_{b,a},H)^{1/(a+b)}.
\]
Both the entropy proof \cite{Kahn01,GT04} and the H\"older's inequality proof \cite{LZ15} (Section~\ref{sec:prod}) extend to biregular graphs. The occupancy method proof \cite{DJPR1} for independent sets (Section~\ref{sec:occup}) also extends to the biregular setting, though one should use two different fugacity parameters for the two vertex parts.

\subsection{Graphs with given degree profile}

Kahn~\cite{Kahn01} made the following conjecture extending Theorem~\ref{thm:kahn} to irregular graphs. We write $d_u$ for the degree of vertex $u \in V(G)$.

\begin{conjecture}[\cite{Kahn01}] \label{conj:kahn-irreg}
  For any graph $G$,
  \[
  i(G) \le \prod_{uv \in E(G)} i(K_{d_u,d_v})^{1/d_ud_v} = \prod_{uv \in E(G)} (2^{d_u} + d^{d_v} - 1)^{1/(d_ud_v)}.
  \]
\end{conjecture}

By the bipartite reduction in Section~\ref{sec:swap}, it suffices to prove the conjecture for bipartite graphs $G$. Galvin and I \cite{GZ11} proved Conjecture~\ref{conj:kahn-irreg} for all $G$ with maximum degree at most $5$.

The following conjecture, due to Galvin~\cite{Gal06}\footnote{A bipartite assumption on $G$ is missing in \cite[Conjecture 1.5]{Gal06}.}, extends Theorem~\ref{thm:GT} and the bipartite case of Conjecture~\ref{conj:kahn-irreg}.

\begin{conjecture}
  For any bipartite graph $G$ and loop-graph $H$,
  \[
  \hom(G,H) \le \prod_{uv \in E(G)} \hom(K_{d_u,d_v},H)^{1/(d_u d_v)}.
  \]
\end{conjecture}

\subsection{Graphs with additional local constraints}

We saw in Theorem~\ref{thm:zhao} and Theorem~\ref{thm:ind-min} that the maximum and minimum of $i(G)^{1/v(G)}$ among $d$-regular graphs $G$ are attained by $K_{d,d}$ and $K_{d+1}$ respectively. What if we impose additional ``local'' constraints to disallow $K_{d,d}$ and $K_{d+1}$? For example, consider the following.
\begin{itemize}
\item What is the infimum of $i(G)^{1/v(G)}$ among $d$-regular triangle-free graphs $G$?
\item What is the supremum of $i(G)^{1/v(G)}$ among $d$-regular graphs $G$ that do not contain any cycles of length 4?
\end{itemize}
These two questions were recently answered by Perarnau and Perkins~\cite{PP}.

\begin{theorem} \label{thm:ind-girth}
(a) Among 3-regular triangle-free graphs $G$, the quantity $i(G)^{1/v(G)}$ is minimized when $G$ is the Petersen graph.

(b) Among 3-regular graphs $G$ without cycles of length 4, the quantity $i(G)^{1/v(G)}$ is maximized when $G$ is the Heawood graph.
\end{theorem}

\begin{center}
\begin{tabular}{ccc}
\begin{tikzpicture}[scale=.6,P/.style={draw, circle, black, fill, inner sep = 0pt, minimum width = 5pt}]
\foreach \x in {0,1,2,3,4}{
	\node[P] (a\x) at ($(\x * 360 / 5 + 90:1)$) {};
	\node[P] (b\x) at ($(\x * 360 / 5 + 90:2)$) {};
	\draw (a\x)--(b\x);
}
\draw (b0)--(b1)--(b2)--(b3)--(b4)--(b0);
\draw (a0)--(a2)--(a4)--(a1)--(a3)--(a0);
\end{tikzpicture} 
& \hspace{6em} &
\begin{tikzpicture}[scale=.6,P/.style={draw, circle, black, fill, inner sep = 0pt, minimum width = 5pt}]
\foreach \x in {0,...,13}{
    \pgfmathsetlengthmacro{\theta}{(\x-.5) * 360 / 14 + 90}
	\node[P] (\x) at ({\theta}:2) {};
}
\draw (0)--(1)--(2)--(3)--(4)--(5)--(6)--(7)--(8)--(9)--(10)--(11)--(12)--(13)--(0);
\draw (1)--(6) (2)--(11) (3)--(8) (4)--(13) (5)--(10) (7)--(12) (9)--(0);
\end{tikzpicture}
\\
Petersen graph & & Heawood graph
\end{tabular}
\end{center}

Theorem~\ref{thm:ind-girth} was proved using the occupancy method discussed in Section~\ref{sec:occup}. The following general problem is very much open.

\begin{problem} \label{prb:local-constraints}
  Let $d \ge 3$ be an integer and $\cF$ be a finite list of graphs. Determine the infimum and supremum of $i(G)^{1/v(G)}$ among $d$-regular graphs $G$ that do not contain any element of $\cF$ as an induced subgraph.
\end{problem}

We pose the following (fairly bold) conjecture that the extrema are always attained by finite graphs. 

\begin{conjecture}[Local constraints imply bounded extrema]
  Let $d \ge 3$ be an integer and $\cF$ be a finite list of graphs. Let $\cG_d(\cF)$ denote the set of finite $d$-regular graphs that do not contain any element of $\cF$ as an induced subgraph. Then there exist $G_{\min}, G_{\max} \in \cG_d(\cF)$ such that for all $G \in \cG_d(\cF)$,
  \[
  i(G_{\min})^{1/v(G_{\min})} \le i(G)^{1/v(G)} \le i(G_{\max})^{1/v(G_{\max})}.
  \]
\end{conjecture}

It would be interesting to know which graphs can arise as extremal graphs in this manner. On the other hand, imposing bipartiteness induces a very different behavior (Section~\ref{sec:min}). See \cite{CR16ar,DJPR2,PP} for discussions of related results and conjectures.

\subsection{Graphs with a given number of vertices and edges}

Let $V(G) =\{1, 2, \dots, n\}$. Let $L_{n,m}$ denote the graph on $n$ vertices obtained by including the first $m$ edges in lexicographic order, i.e., $12, 13, \dots, 1n, 23,24,\dots$. Recall that $i_t(G)$ is the number of independent sets of size $t$ in $G$. The following result is a consequence of the Kruskal--Katona theorem \cite{Kru63,Kat68}.

\begin{theorem}
  For any graph $G$ with $n$ vertices and $m$ edges, and positive integer $t$, one has $i_t(G) \le i_t(L_{n,m})$.
\end{theorem}

Reiher's clique density theorem~\cite{Rei16} solves the corresponding minimization problem, which is significantly more difficult. See \cite{LPS10,MN15} and their references for results and conjectures on the analogous problem of maximizing the number of proper $q$-colorings in a graph with a given number of vertices and edges, and \cite{CK,CR11,CR14jgt} for graph homomorphisms.

\subsection{Minimum degree condition}
What if we relax the $d$-regular condition in Theorem~\ref{thm:zhao} to minimum degree $d$? The following result was conjectured by Galvin~\cite{Gal11} and proved by Cutler and Radcliffe~\cite{CR14}.

\begin{theorem}[\cite{CR14}]
  Let $\delta \le n/2$. Let $G$ be an $n$-vertex graph with minimum degree at least $\delta$. Then $i(G) \le i(K_{\delta,n-\delta})$.
\end{theorem}

More generally, for any $\delta < n$, write $n = a(n-\delta)+b$ with $a$ and $b$ nonnegative integers and $b < n-\delta$,  one has $i(G) \le i(\ol{a K_{n-\delta} \cup K_b}) = a(2^{n-\delta} - 1) + 2^b$ for any graph $G$ on $n$ vertices with minimum degree at least $\delta$. Here $\ol{G}$ denotes the edge-complement of $G$.

The following strengthening was conjectured by Engbers and Galvin~\cite{EG14}. It was proved by Alexander, Cutler, and Mink~\cite{ACM12} for bipartite graphs, and proved by Gan, Loh, and Sudakov~\cite{GLS15} in general. Recall that $i_t(G)$ is the number of independent sets of size $t$ in $G$.

\begin{theorem}[\cite{GLS15}]
  Let $\delta \le n/2$ and $t \ge 3$. Let $G$ be an $n$-vertex graph with minimum degree at least $\delta$. Then $i_t(G) \le i_t(K_{\delta,n-\delta})$.
\end{theorem}

Note that this claim is false for $t = 2$. 
See \cite{CK,CR11,CR14jgt,Eng15} for discussions on the analogous problem of maximizing the number of homomorphisms into a fixed $H$.

\subsection{Matchings}

Let $m(G)$ denote the number of matchings in a graph $G$, $m_t(G)$ the number of matchings with $t$ edges in $G$, and $pm(G) := m_{v(G)/2}(G)$ the number of perfect matchings in $G$.

The following upper bounds on $m(G)$ and $pm(G)$ have a curious semblance to Theorems~\ref{thm:kahn} and \ref{thm:zhao}. The quantity $m(G)$ for matchings can be viewed as analogous to $i(G)$ for independent sets.

For the number of perfect matchings, the bipartite case was conjectured in 1963 by Minc~\cite{Minc63} and proved by Br\`egman~\cite{Bre73} a decade later. Many different proofs have been given since then. The non-bipartite extension is due to Kahn and Lov\'asz (unpublished). See \cite{Gal}, which includes a statement allowing irregular $G$.

\begin{theorem}\label{thm:pm}
  For any $d$-regular graph $G$,
  \[
  pm(G)^{1/v(G)} \le pm(K_{d,d})^{1/(2d)} = (d!)^{1/(2d)}.
  \]
\end{theorem}

The occupancy method was used in \cite{DJPR1} to give an alternative proof of Theorem~\ref{thm:pm}, along with a new upper bound on $m(G)$, as well as a weighted extension analogous to Theorem~\ref{thm:indep-poly}. Define the \emph{matching polynomial}
\[
M_G(\lambda) := \sum_{M \in \mathcal{M}(G)} \lambda^{|M|}
\]
where $\mathcal{M}(G)$ is the set of matchings in $G$, and $|M|$ is the number of edges in the matching $M$. 

\begin{theorem}[\cite{DJPR1}] \label{thm:matching-weighted}
 	For any $d$-regular graph $G$ and $\lambda \ge 0$,
 	\[
 	M_G(\lambda)^{1/v(G)} \le M_{K_{d,d}}(\lambda)^{1/(2d)}.
 	\]
 	In particular, setting $\lambda = 1$ yields
 	\[
 	m(G)^{1/v(G)} \le m(K_{d,d})^{1/(2d)}.
 	\]
\end{theorem}

In fact, an edge occupancy fraction result analogous to Theorem~\ref{thm:occupancy} holds. We refer to \cite[Theorem~3]{DJPR1} for the exact statement.
Note that setting $\lambda \to \infty$ in Theorem~\ref{thm:matching-weighted} recovers Theorem~\ref{thm:pm}, since the dominant term in $M_G(\lambda)$ is $pm(G)\lambda^{v(G)/2}$. 

The following matching analog of Conjecture~\ref{conj:ind-fixed-size} remains open. See \cite{CGT09,DJPR1} for discussion.

\begin{conjecture}[\cite{FKM08}] \label{conj:mat-fixed-size}
  If $G$ is an $2ad$-vertex $d$-regular graph and $t \ge 0$, then $m_t(G) \le m_t(a K_{d,d})$.
\end{conjecture}

The infimum of $pm(G)^{1/v(G)}$ for $d$-regular bipartite graphs $G$ is well understood. The infimum corresponds to random bipartite graphs $G$.

\begin{theorem}[Voorhoeve~\cite{Voo79} for $d=3$ and Schrijver~\cite{Sch98} for all $d$]
  If $G$ is a $d$-regular bipartite graph on $2n$ vertices, then
\[
pm(G)^{1/n} \ge \frac{(d-1)^{d-1}}{d^{d-2}}.
\]
\end{theorem}

See \cite{LS10} for an exposition. The corresponding minimization problem for $m(G)$, and more generally for $m_t(G)$ and $M_G(\lambda)$, was solved by Gurvits~\cite{Gur} and extended by Csikv\'ari \cite{Csi}.

\section*{Acknowledgments}

I am grateful to Joe Gallian for the REU opportunity in 2009 where I began working on this problem (resulting in \cite{Zhao10}). I thank P\'eter Csikv\'ari, David Galvin, Joonkyung Lee, Will Perkins, and Prasad Tetali for carefully reading a draft of this paper and providing helpful comments. I also thank the anonymous reviewers for suggestions that improved the exposition of the paper.

\bibliographystyle{amsplain_mod2}
\bibliography{ref_ext_reg}

\end{document}

%% file: fig-bst-loop-path.tex
\begin{tikzpicture}[scale=.6]

\begin{scope}
	\begin{scope}
		\foreach \i in {0,...,7}{
			\node[P] (\i) at (0,\i) {};
		}
		\draw (7)--(6)--(5)--(4)--(3)--(2)--(1)--(0) edge[-,in = 135, out = 225, loop, distance=1cm] (0);
		\node at (0,-1) {$H$};
	\end{scope}
	
	\begin{scope}[xshift=1.5cm]
		\node at (3.5,-1) {$H^{\bst}$};
		\node[B] (7 3) at (7,3) {};
		\node[W] (4 7) at (4,7) {};
		\node[W] (1 3) at (1,3) {};
		\node[W] (6 4) at (6,4) {};
		\node[B] (3 0) at (3,0) {};
		\node[B] (5 4) at (5,4) {};
		\node[W] (0 7) at (0,7) {};
		\node[B] (5 6) at (5,6) {};
		\node[B] (2 6) at (2,6) {};
		\node[B] (1 6) at (1,6) {};
		\node[B] (5 1) at (5,1) {};
		\node[W] (3 7) at (3,7) {};
		\node[W] (2 5) at (2,5) {};
		\node[W] (0 3) at (0,3) {};
		\node[B] (7 2) at (7,2) {};
		\node[W] (4 0) at (4,0) {};
		\node[B] (1 2) at (1,2) {};
		\node[W] (6 7) at (6,7) {};
		\node[B] (3 3) at (3,3) {};
		\node[W] (2 0) at (2,0) {};
		\node[B] (7 6) at (7,6) {};
		\node[B] (4 4) at (4,4) {};
		\node[W] (6 3) at (6,3) {};
		\node[W] (1 5) at (1,5) {};
		\node[B] (3 6) at (3,6) {};
		\node[B] (2 2) at (2,2) {};
		\node[B] (7 7) at (7,7) {};
		\node[W] (5 7) at (5,7) {};
		\node[B] (5 3) at (5,3) {};
		\node[W] (4 1) at (4,1) {};
		\node[B] (1 1) at (1,1) {};
		\node[W] (2 7) at (2,7) {};
		\node[B] (3 2) at (3,2) {};
		\node[B] (0 0) at (0,0) {};
		\node[B] (6 6) at (6,6) {};
		\node[B] (5 0) at (5,0) {};
		\node[B] (7 1) at (7,1) {};
		\node[W] (4 5) at (4,5) {};
		\node[B] (0 4) at (0,4) {};
		\node[B] (5 5) at (5,5) {};
		\node[B] (1 4) at (1,4) {};
		\node[W] (6 0) at (6,0) {};
		\node[B] (7 5) at (7,5) {};
		\node[W] (2 3) at (2,3) {};
		\node[W] (2 1) at (2,1) {};
		\node[W] (4 2) at (4,2) {};
		\node[B] (1 0) at (1,0) {};
		\node[W] (6 5) at (6,5) {};
		\node[W] (3 5) at (3,5) {};
		\node[W] (0 1) at (0,1) {};
		\node[B] (7 0) at (7,0) {};
		\node[B] (4 6) at (4,6) {};
		\node[B] (5 2) at (5,2) {};
		\node[W] (6 1) at (6,1) {};
		\node[B] (3 1) at (3,1) {};
		\node[B] (0 2) at (0,2) {};
		\node[B] (7 4) at (7,4) {};
		\node[B] (0 6) at (0,6) {};
		\node[W] (6 2) at (6,2) {};
		\node[W] (4 3) at (4,3) {};
		\node[W] (1 7) at (1,7) {};
		\node[W] (0 5) at (0,5) {};
		\node[B] (3 4) at (3,4) {};
		\node[B] (2 4) at (2,4) {};
		\draw (7 3) -- (6 4);
		\draw (7 3) -- (6 2);
		\draw (4 7) -- (5 6);
		\draw (4 7) -- (3 6);
		\draw (1 3) -- (2 4);
		\draw (1 3) -- (0 2);
		\draw (1 3) -- (0 4);
		\draw (6 4) -- (7 5);
		\draw (6 4) -- (5 3);
		\draw (3 0) -- (2 0);
		\draw (3 0) -- (4 1);
		\draw (3 0) -- (2 1);
		\draw (3 0) -- (4 0);
		\draw (5 4) -- (4 5);
		\draw (5 4) -- (6 3);
		\draw (5 4) -- (6 5);
		\draw (5 4) -- (4 3);
		\draw (0 7) -- (0 6);
		\draw (0 7) -- (1 6);
		\draw (5 6) -- (4 5);
		\draw (5 6) -- (6 7);
		\draw (5 6) -- (6 5);
		\draw (2 6) -- (1 5);
		\draw (2 6) -- (3 5);
		\draw (2 6) -- (3 7);
		\draw (2 6) -- (1 7);
		\draw (1 6) -- (2 7);
		\draw (1 6) -- (2 5);
		\draw (1 6) -- (0 5);
		\draw (5 1) -- (4 2);
		\draw (5 1) -- (6 2);
		\draw (5 1) -- (6 0);
		\draw (5 1) -- (4 0);
		\draw (3 7) -- (4 6);
		\draw (2 5) -- (3 4);
		\draw (2 5) -- (3 6);
		\draw (2 5) -- (1 4);
		\draw (0 3) -- (1 2);
		\draw (0 3) -- (0 2);
		\draw (0 3) -- (1 4);
		\draw (0 3) -- (0 4);
		\draw (7 2) -- (6 3);
		\draw (7 2) -- (6 1);
		\draw (4 0) -- (3 1);
		\draw (4 0) -- (5 0);
		\draw (1 2) -- (0 1);
		\draw (1 2) -- (2 3);
		\draw (1 2) -- (2 1);
		\draw (6 7) -- (7 6);
		\draw (2 0) -- (1 0);
		\draw (2 0) -- (3 1);
		\draw (7 6) -- (6 5);
		\draw (6 3) -- (7 4);
		\draw (6 3) -- (5 2);
		\draw (1 5) -- (0 6);
		\draw (1 5) -- (2 4);
		\draw (1 5) -- (0 4);
		\draw (3 6) -- (2 7);
		\draw (3 6) -- (4 5);
		\draw (5 7) -- (4 6);
		\draw (5 3) -- (4 2);
		\draw (5 3) -- (6 2);
		\draw (4 1) -- (3 2);
		\draw (4 1) -- (5 2);
		\draw (4 1) -- (5 0);
		\draw (3 2) -- (2 3);
		\draw (3 2) -- (4 3);
		\draw (3 2) -- (2 1);
		\draw (5 0) -- (6 1);
		\draw (5 0) -- (6 0);
		\draw (7 1) -- (6 2);
		\draw (7 1) -- (6 0);
		\draw (4 5) -- (3 4);
		\draw (0 4) -- (0 5);
		\draw (1 4) -- (2 3);
		\draw (1 4) -- (0 5);
		\draw (6 0) -- (7 0);
		\draw (2 3) -- (3 4);
		\draw (2 1) -- (1 0);
		\draw (4 2) -- (3 1);
		\draw (1 0) -- (0 1);
		\draw (6 5) -- (7 4);
		\draw (3 5) -- (2 4);
		\draw (3 5) -- (4 6);
		\draw (0 1) -- (0 2);
		\draw (7 0) -- (6 1);
		\draw (5 2) -- (6 1);
		\draw (5 2) -- (4 3);
		\draw (0 6) -- (0 5);
		\draw (0 6) -- (1 7);
		\draw (4 3) -- (3 4);
	\end{scope}
\end{scope}

\begin{scope}[xshift=12cm]
	\begin{scope}
		\foreach \i in {0,...,7}{
			\node[P] (\i) at (0,\i) {};
		}
		\draw (7)--(6)--(5)--(4)--(3)--(2)--(1)  edge[-,in = 135, out = 225, loop, distance=1cm] (1)--(0);
		\node at (0,-1) {$H$};
	\end{scope}
	
	\begin{scope}[xshift=1.5cm]
		\node at (3.5,-1) {$H^{\bst}$};
		\node[B] (7 3) at (7,3) {};
		\node[W] (4 7) at (4,7) {};
		\node[W] (1 3) at (1,3) {};
		\node[W] (6 4) at (6,4) {};
		\node[B] (3 0) at (3,0) {};
		\node[B] (5 4) at (5,4) {};
		\node[W] (0 7) at (0,7) {};
		\node[B] (5 6) at (5,6) {};
		\node[B] (2 6) at (2,6) {};
		\node[B] (1 6) at (1,6) {};
		\node[B] (5 1) at (5,1) {};
		\node[W] (3 7) at (3,7) {};
		\node[W] (2 5) at (2,5) {};
		\node[W] (0 3) at (0,3) {};
		\node[B] (7 2) at (7,2) {};
		\node[W] (4 0) at (4,0) {};
		\node[B] (1 2) at (1,2) {};
		\node[W] (6 7) at (6,7) {};
		\node[B] (3 3) at (3,3) {};
		\node[W] (2 0) at (2,0) {};
		\node[B] (7 6) at (7,6) {};
		\node[B] (4 4) at (4,4) {};
		\node[W] (6 3) at (6,3) {};
		\node[W] (1 5) at (1,5) {};
		\node[B] (3 6) at (3,6) {};
		\node[B] (2 2) at (2,2) {};
		\node[B] (7 7) at (7,7) {};
		\node[W] (5 7) at (5,7) {};
		\node[B] (5 3) at (5,3) {};
		\node[W] (4 1) at (4,1) {};
		\node[B] (1 1) at (1,1) {};
		\node[W] (2 7) at (2,7) {};
		\node[B] (3 2) at (3,2) {};
		\node[B] (0 0) at (0,0) {};
		\node[B] (6 6) at (6,6) {};
		\node[B] (5 0) at (5,0) {};
		\node[B] (7 1) at (7,1) {};
		\node[W] (4 5) at (4,5) {};
		\node[B] (0 4) at (0,4) {};
		\node[B] (5 5) at (5,5) {};
		\node[B] (1 4) at (1,4) {};
		\node[W] (6 0) at (6,0) {};
		\node[B] (7 5) at (7,5) {};
		\node[W] (2 3) at (2,3) {};
		\node[W] (2 1) at (2,1) {};
		\node[W] (4 2) at (4,2) {};
		\node[B] (1 0) at (1,0) {};
		\node[W] (6 5) at (6,5) {};
		\node[W] (3 5) at (3,5) {};
		\node[W] (0 1) at (0,1) {};
		\node[B] (7 0) at (7,0) {};
		\node[B] (4 6) at (4,6) {};
		\node[B] (5 2) at (5,2) {};
		\node[W] (6 1) at (6,1) {};
		\node[B] (3 1) at (3,1) {};
		\node[B] (0 2) at (0,2) {};
		\node[B] (7 4) at (7,4) {};
		\node[B] (0 6) at (0,6) {};
		\node[W] (6 2) at (6,2) {};
		\node[W] (4 3) at (4,3) {};
		\node[W] (1 7) at (1,7) {};
		\node[W] (0 5) at (0,5) {};
		\node[B] (3 4) at (3,4) {};
		\node[B] (2 4) at (2,4) {};
		\draw (7 3) -- (6 4);
		\draw (7 3) -- (6 2);
		\draw (4 7) -- (5 6);
		\draw (4 7) -- (3 6);
		\draw (1 3) -- (1 2);
		\draw (1 3) -- (2 4);
		\draw (1 3) -- (0 2);
		\draw (1 3) -- (1 4);
		\draw (1 3) -- (0 4);
		\draw (6 4) -- (7 5);
		\draw (6 4) -- (5 3);
		\draw (3 0) -- (4 1);
		\draw (3 0) -- (2 1);
		\draw (5 4) -- (4 5);
		\draw (5 4) -- (6 3);
		\draw (5 4) -- (6 5);
		\draw (5 4) -- (4 3);
		\draw (0 7) -- (1 6);
		\draw (5 6) -- (4 5);
		\draw (5 6) -- (6 7);
		\draw (5 6) -- (6 5);
		\draw (2 6) -- (1 5);
		\draw (2 6) -- (3 5);
		\draw (2 6) -- (3 7);
		\draw (2 6) -- (1 7);
		\draw (1 6) -- (2 7);
		\draw (1 6) -- (1 5);
		\draw (1 6) -- (0 5);
		\draw (1 6) -- (1 7);
		\draw (1 6) -- (2 5);
		\draw (5 1) -- (6 1);
		\draw (5 1) -- (6 0);
		\draw (5 1) -- (6 2);
		\draw (5 1) -- (4 2);
		\draw (5 1) -- (4 1);
		\draw (5 1) -- (4 0);
		\draw (3 7) -- (4 6);
		\draw (2 5) -- (3 4);
		\draw (2 5) -- (3 6);
		\draw (2 5) -- (1 4);
		\draw (0 3) -- (1 2);
		\draw (0 3) -- (1 4);
		\draw (7 2) -- (6 3);
		\draw (7 2) -- (6 1);
		\draw (4 0) -- (3 1);
		\draw (1 2) -- (0 1);
		\draw (1 2) -- (2 3);
		\draw (1 2) -- (2 1);
		\draw (6 7) -- (7 6);
		\draw (2 0) -- (3 1);
		\draw (7 6) -- (6 5);
		\draw (6 3) -- (7 4);
		\draw (6 3) -- (5 2);
		\draw (1 5) -- (1 4);
		\draw (1 5) -- (0 6);
		\draw (1 5) -- (0 4);
		\draw (1 5) -- (2 4);
		\draw (3 6) -- (2 7);
		\draw (3 6) -- (4 5);
		\draw (5 7) -- (4 6);
		\draw (5 3) -- (4 2);
		\draw (5 3) -- (6 2);
		\draw (4 1) -- (3 2);
		\draw (4 1) -- (3 1);
		\draw (4 1) -- (5 0);
		\draw (4 1) -- (5 2);
		\draw (3 2) -- (2 3);
		\draw (3 2) -- (4 3);
		\draw (3 2) -- (2 1);
		\draw (5 0) -- (6 1);
		\draw (7 1) -- (6 1);
		\draw (7 1) -- (6 2);
		\draw (7 1) -- (6 0);
		\draw (4 5) -- (3 4);
		\draw (1 4) -- (2 3);
		\draw (1 4) -- (0 5);
		\draw (2 3) -- (3 4);
		\draw (2 1) -- (1 0);
		\draw (2 1) -- (3 1);
		\draw (4 2) -- (3 1);
		\draw (1 0) -- (0 1);
		\draw (6 5) -- (7 4);
		\draw (3 5) -- (2 4);
		\draw (3 5) -- (4 6);
		\draw (7 0) -- (6 1);
		\draw (5 2) -- (6 1);
		\draw (5 2) -- (4 3);
		\draw (0 6) -- (1 7);
		\draw (4 3) -- (3 4);
	\end{scope}
\end{scope}

\end{tikzpicture}

%% file: fig-non-bst.tex
\begin{tikzpicture}[scale=.6]

	\begin{scope}
		\foreach \i in {0,...,4}{
			\node[P] (\i) at (0,\i) {};
		}
		\draw (4)--(3)--(2) edge[-,in = 135, out = 225, loop, distance=1cm] (2)--(1)--(0);
		\node at (0,-1) {$H$};
	\end{scope}
	
	\begin{scope}[xshift=1.5cm]
		\node at (2,-1) {$H^{\bst}$};
		\node[P] (1 3) at (1,3) {};
		\node[P] (3 0) at (3,0) {};
		\node[P] (2 1) at (2,1) {};
		\node[P] (0 3) at (0,3) {};
		\node[P] (4 0) at (4,0) {};
		\node[P] (1 2) at (1,2) {};
		\node[P] (3 3) at (3,3) {};
		\node[P] (4 4) at (4,4) {};
		\node[P] (2 2) at (2,2) {};
		\node[P] (4 1) at (4,1) {};
		\node[P] (1 1) at (1,1) {};
		\node[P] (3 2) at (3,2) {};
		\node[P] (0 0) at (0,0) {};
		\node[P] (0 4) at (0,4) {};
		\node[P] (1 4) at (1,4) {};
		\node[P] (2 3) at (2,3) {};
		\node[P] (4 2) at (4,2) {};
		\node[P] (1 0) at (1,0) {};
		\node[P] (0 1) at (0,1) {};
		\node[P] (3 1) at (3,1) {};
		\node[P] (2 4) at (2,4) {};
		\node[P] (2 0) at (2,0) {};
		\node[P] (4 3) at (4,3) {};
		\node[P] (3 4) at (3,4) {};
		\node[P] (0 2) at (0,2) {};
		\draw (1 3) -- (2 4);
		\draw (1 3) -- (0 2);
		\draw (1 3) -- (0 4);
		\draw (3 0) -- (4 1);
		\draw (3 0) -- (2 1);
		\draw (2 1) -- (1 2);
		\draw (2 1) -- (2 0);
		\draw (2 1) -- (1 0);
		\draw (2 1) -- (3 2);
		\draw (0 3) -- (1 2);
		\draw (0 3) -- (1 4);
		\draw (4 0) -- (3 1);
		\draw (1 2) -- (0 1);
		\draw (1 2) -- (2 3);
		\draw (1 2) -- (0 2);
		\draw (4 1) -- (3 2);
		\draw (3 2) -- (2 3);
		\draw (3 2) -- (4 3);
		\draw (3 2) -- (4 2);
		\draw (1 4) -- (2 3);
		\draw (2 3) -- (3 4);
		\draw (2 3) -- (2 4);
		\draw (4 2) -- (3 1);
		\draw (1 0) -- (0 1);
		\draw (3 1) -- (2 0);
		\draw (4 3) -- (3 4);
		\draw[ultra thick] (0 2) -- (1 2) -- (2 3) -- (2 4) -- (1 3) --  (0 2);
	\end{scope}

\end{tikzpicture}